\documentclass[10pt]{amsart}
\usepackage[utf8]{inputenc}
\usepackage{graphics}
\usepackage[all, cmtip]{xy}
\usepackage{amsthm,amsfonts,amssymb,amsmath,amsxtra}
\usepackage[all]{xy}
\usepackage{flexisym}
\usepackage{mathrsfs}
\usepackage{enumitem}
\usepackage[colorlinks]{hyperref}
\hypersetup{
  citecolor   = blue 
 }

\usepackage{tikz-cd}
\usepackage{blkarray}
\newcommand{\matindex}[1]{\mbox{\scriptsize#1}}

\newcommand{\xdownarrow}[1]{%
  {\left\downarrow\vbox to #1{}\right.\kern-\nulldelimiterspace}
}
\newcommand{\quash}[1]{}  

\newcommand{\M}{{\rm M}}
\newcommand{\loc}{{\rm loc}}
\newcommand{\spl}{{\rm spl}}

\newcommand{\ZZ}{\mathbb{Z}}

\newcommand{\calF}{\mathcal{F}}
\newcommand{\calL}{\mathcal{L}}

\newcommand{\calU}{\mathcal{U}}
\newcommand{\calN}{\mathcal{N}}
\newcommand{\calZ}{\mathcal{Z}}
\newcommand{\calY}{\mathcal{Y}}

\newcommand{\scrG}{\mathscr{G}}

\newcommand{\bb }{\langle}
\newcommand{\pp}{\rangle}

\newcommand{\rightarroweq}{\stackrel{\sim}{\rightarrow}}
\newcommand{\br}{\breve}

\usepackage{multicol}
\numberwithin{equation}{subsection}

\newtheorem{Theorem}{Theorem}[section]
\newtheorem{Remark}[Theorem]{Remark}
\newtheorem{Remarks}[Theorem]{Remarks}
\newtheorem{Lemma}[Theorem]{Lemma}
\newtheorem{Proposition}[Theorem]{Proposition}
\newtheorem{Corollary}[Theorem]{Corollary}

\newtheorem{Main Theorem}[Theorem]{Main Theorem}
\newtheorem{Definition}[Theorem]{Definition}

\usepackage[colorlinks]{hyperref}
\makeatletter
\renewcommand*\env@matrix[1][*\c@MaxMatrixCols c]{%
  \hskip -\arraycolsep
  \let\@ifnextchar\new@ifnextchar
  \array{#1}}
\makeatother

\newif\ifgrading

\gradingfalse

\usepackage[colorlinks]{hyperref}
\hypersetup{linkcolor=black}
\usepackage{xcolor}
\usepackage{ wasysym }

\usepackage{tikz-cd}

\newcommand{\calG}{{\mathcal G}}

\newcommand{\Spec}{{\rm Spec \, } }

\usepackage{xcolor}

\newcommand{\Addresses}{{
		\bigskip
		\footnotesize
		\textsc{Department of Mathematics, Universität Münster, Münster, 48149, Germany}\par\nopagebreak
		\textit{E-mail address:} \texttt{io.zachos@uni-muenster.de}\\
		
	\textsc{Department of Applied Mathematics, University of Science and Technology Beijing,
			Beijing, 100083, China}\par\nopagebreak
		\textit{E-mail address:} \texttt{zhihaozhao@ustb.edu.cn}
}}	

\begin{document}
\title[Unitary splitting Rapoport--Zink spaces]{The basic locus of unitary splitting Rapoport--Zink spaces with vertex stabilizer level}
	\date{}
	\author{I. Zachos and Z. Zhao}
		
	\begin{abstract}
 We construct the Bruhat-Tits stratification of the ramified unitary splitting Rapoport-Zink
space, with the level being the stabilizer of a vertex lattice. To determine certain local properties of the Bruhat-Tits strata, we develop a theory of the strata splitting models. To study their global structure, we establish an explicit isomorphism between the Bruhat-Tits strata and certain (modified) Deligne-Lusztig varieties.
	\end{abstract}
\maketitle	
	\tableofcontents
\section{Introduction}\label{Intro}
\subsection{}
This paper contributes to the theory of integral models of Shimura varieties by providing a concrete description of the reduced basic locus of certain ramified unitary Rapoport-Zink (RZ) spaces at a maximal vertex level with signature $ (n-1,1)$. The study of the basic locus has a long history, and various cases of both orthogonal and unitary RZ spaces have been considered. This has led to many important applications in number theory. 
 For example, it has found applications in Kudla-Rapoport conjecture, 
 which relates arithmetic intersection numbers of special cycles on Shimura varieties to Eisenstein series (see \cite{CZ1,CZ2,HLSY,LL,Yao,Luo2}). It has also played a role in the arithmetic Gan–Gross–Prasad conjecture, the arithmetic fundamental lemma conjecture and the arithmetic transfer conjecture (see, for example,  \cite{RSZ,Zhang,LMZ,LRZ}).



In the orthogonal case, the reduced basic locus of the RZ space was first studied by Howard-Pappas \cite{HP} in the self-dual case, and subsequently by Oki \cite{Oki} in the almost self-dual case. More recently, He-Zhou \cite{HZ} generalized these results by treating all maximal level structures. 

For the general unitary group ${\rm GU}(n-1,1)$, the basic locus of the RZ space was first studied by Vollaard \cite{Vol10} and Vollaard-Wedhorn \cite{VW} at an inert prime with hyperspecial level. This was later extended by Cho \cite{Cho} to all maximal parahoric level structures and more recently by Muller \cite{Muller} for arbitrary parahoric level. In the ramified case, the basic locus was studied by Rapoport-Terstiege-Wilson \cite{RTW} at self-dual levels, by Wu \cite{Wu} for the exotic smooth cases, and more recently by He-Luo-Shi \cite{HLS2}, who extended the results to all maximal vertex levels. 

Roughly speaking, in all the above cases the following picture emerges: the basic locus admits a stratification by (generalized) Deligne-Lusztig varieties and the intersection pattern of the strata can be described in terms of some Bruhat-Tits building. This is the so-called \textit{Bruhat-Tits (BT) stratification}. In the work of G\"ortz-He-Nie \cite{GHN}, the authors give a complete classification of the Shimura varieties whose basic locus admits a BT-stratification. 

For the unitary ramified case, which is the case we are interested in, the RZ spaces have an explicit moduli description and this is a key tool for the study of their basic locus. These RZ spaces have bad reduction at all maximal vertex levels, except for the two exotic smooth cases studied in \cite{Wu}. 
For the self-dual and almost $\pi$-modular cases, in order to resolve the singularities, variations of this moduli problem were obtained by Krämer \cite{Kr} and Richarz \cite{Richarz}, respectively, in which they added to the moduli problem an additional linear datum of a flag of the Hodge filtration with certain restrictive properties. This construction was then generalized in our work \cite{ZacZhao2} to all maximal vertex lattices. These are the \textit{splitting Rapoport–Zink spaces}. 

In the work of He-Li-Shi-Yang \cite{HLSY}, the basic locus was studied for the Krämer model. In this paper, we generalize the results of \cite{HLSY} and give a concrete description of the reduced basic locus of the splitting RZ space for any maximal vertex lattice. More precisely, we describe the BT stratification of the basic locus and show that each BT stratum is isomorphic to a certain (modified) Deligne-Lusztig variety. 
Moreover, to study local properties of these BT strata—such as normality, dimension, and reducedness—
inspired by \cite{HLS2}, we develop a theory of \textit{strata splitting models}, which are simpler schemes defined by purely linear-algebraic data, and we prove that the BT strata are \'etale locally isomorphic to 
these models. 

We hope that the theory of strata splitting models will be a useful tool in the study of the reduced basic locus of splitting RZ spaces of higher signatures, such as those considered by Hernandez-Bijakowski \cite{BH}, in our work \cite{ZacZhao1} and in our joint work with Bijakowski \cite{BZZ}. Finally, we point out that little is known about splitting models for more general (quasi-)parahoric levels, and we intend to pursue this direction in future work. We anticipate that the strata splitting models introduced here will provide a useful tool for studying the corresponding basic loci at deeper level structures.

\subsection{}
Let us give some details. To explain our results, we begin by introducing some notation. Let $F/F_0$ be a ramified extension of $p$-adic fields, 
where $p$ is an odd prime, with residue field $k$ and uniformizers $\pi$ and $\pi_0$ respectively, satisfying $\pi^2=\pi_0$. Let $\bar{k}$ be a fixed algebraic closure of $k$. Denote by $\breve F$ the completion of the maximal unramified extension of $F$ and let $O_F$, $O_{\breve F}$ be the ring of integers of $F$, $\breve F$ respectively. Let $h, n$ be integers with $0 \leq h <\frac{n}{2} $. (Note that the choice of $h$ depends on the maximal vertex level.)

Fix a supersingular hermitian $O_F$-module $(\mathbb{X},\iota_{\mathbb{X}},\lambda_{\mathbb{X}})$ over $\Spec \bar{k}$ of rank $n$ and type $2h$ (with signature $(n-1, 1)$); this is the framing object and we refer the reader to 
\S \ref{sec2.1} for the precise definition. We define the splitting RZ space $\mathcal{N}_n^{\rm spl}$ of signature $(n-1,1)$ to be the moduli functor that assigns to each \( S \in \mathrm{Nilp}\, O_{\breve{F}} \) the set of isomorphism classes of quintuples $(X, \iota, \lambda, \rho,\operatorname{Fil}^0(X))$ where $(X, \iota, \lambda) $ is a hermitian $O_F$-module over $S$ of dimension $n$ and type $2h$, $ \rho$ is an \( O_F \)-linear quasi-isogeny of height zero from $X$ to the framing object on the special fiber, and $\operatorname{Fil}^0(X)$ is locally a $\mathcal{O}_S$-direct summand of the Hodge filtration $\operatorname{Fil}(X) \subset D(X)$ of rank one that satisfies the \textit{splitting conditions} 
    \[  (\iota(\pi)+\pi) (\operatorname{Fil}(X)) \subset \operatorname{Fil}^0(X) \quad \text{and} \quad (\iota(\pi)-\pi) (\operatorname{Fil}^0(X))=0;
    \]  
see Definition \ref{SplDef} for more details. By the local model diagram, $\mathcal{N}_n^{\rm spl}$ is \'etale locally isomorphic to the splitting model $\M_n^{\spl, [2h]}$ defined in \cite{ZacZhao2}. Thus, $\mathcal{N}_n^{\rm spl}$ is representable by a flat normal formal scheme of relative dimension $n - 1$ over $\operatorname{Spf} O_{\breve F}$.


To the triple $(\mathbb{X},\iota_{\mathbb{X}},\lambda_{\mathbb{X}})$, there exists a hermitian space $C$ over $F$ of dimension $n$. Consider a lattice $\Lambda\subset C$ and its dual lattice $\Lambda^\sharp$ with respect to the hermitian form on $C$. We call $\Lambda\subset C$ a vertex lattice if it satisfies 
\[
\pi \Lambda^{\sharp} \subset \Lambda \subset  \Lambda^{\sharp}.   
\]
We denote by $t(\Lambda):=\dim(\Lambda^\sharp/\Lambda)$ the type of a vertex lattice, which is an even integer (see \S \ref{sec2.2}). By abuse of notation, we will write $2t$ instead of $t(\Lambda)$. 
Let $L_{\mathcal{Z}}$ denote the set of all vertex lattices of type $2t \geq 2h$, and let $L_{\mathcal{Y}}$ denote the set of all vertex lattices of type $2t\leq 2h$. For each $\Lambda_1 \in L_{\mathcal{Z}}$ and $\Lambda_2 \in L_{\mathcal{Y}}$, we define closed subschemes $\mathcal{Z}^{\rm spl}(\Lambda_1)$ and $\mathcal{Y}^{\rm spl}(\Lambda_2^\sharp)$ of the special fiber of the RZ space $\calN_{n}^\spl$ (see \S \ref{BT_strata}).
The first main result of the current work is the following theorem which is proved in \S \ref{BTstrat.}. 
\begin{Theorem}
The Bruhat-Tits stratification of the reduced basic locus of the splitting {\rm RZ} space $\calN_{n, {\rm red}}^\spl$ is
\[
      \calN_{n, {\rm red}}^\spl = \left( \bigcup_{\Lambda_1 \in L_{\mathcal{Z}} }\mathcal{Z}^{\rm spl}(\Lambda_1) \right) \cup \left( \bigcup_{\Lambda_2 \in L_{\mathcal{Y}}} \mathcal{Y}^{\rm spl}(\Lambda_2^\sharp) \right).
\]	
  \begin{enumerate}
  \item  
   These strata satisfy the following inclusion relations:
  \begin{itemize}
    \item[(i)] For any $\Lambda_1, \Lambda_2 \in L_{\mathcal{Z}}$ of type greater than $ 2h$, 
    $  \Lambda_1 \subseteq \Lambda_2$ if and only if $\mathcal{Z}^{\rm spl}(\Lambda_2) \subseteq \mathcal{Z}^{\rm spl}(\Lambda_1)$.    
    \item[(ii)] For any $\Lambda_1, \Lambda_2 \in  L_{\mathcal{Y}}$ of type less than $2h$, 
    $  \Lambda_1 \subseteq \Lambda_2$ if and only if $\mathcal{Y}^{\rm spl}(\Lambda^{\sharp}_1) \subseteq \mathcal{Y}^{\rm spl}(\Lambda^{\sharp}_2)$.  
    \item[(iii)] For any $\Lambda_1\in L_{\mathcal{Z}}$ of type greater than $2h$, $\Lambda_2 \in  L_{\mathcal{Y}}$ of type less than $2h$, $  \Lambda_1 \subseteq \Lambda_2$ if and only if the intersection $\mathcal{Z}^{\rm spl}(\Lambda_1) \cap \mathcal{Y}^{\rm spl}(\Lambda_2^\sharp)$ is non-empty.
  \end{itemize}

  \item In the following, assume that $\Lambda, \Lambda'$ are vertex lattices of type $ 2t$ with $t \neq h$, and $\Lambda_0, \Lambda'_0$ are vertex lattices of type $2t$ with $ t=h$. 
  \begin{itemize}
    \item[(i)] The intersection $\mathcal{Z}^{\rm spl}(\Lambda) \cap \mathcal{Z}^{\rm spl}(\Lambda')$ (resp. $\mathcal{Y}^{\rm spl}(\Lambda^\sharp) \cap \mathcal{Y}^{\rm spl}(\Lambda^{\prime \sharp)}$) is non-empty if and only if $\Lambda'' = \Lambda + \Lambda'$ (resp. $\Lambda''=\Lambda \cap \Lambda'$) is a vertex lattice; in which case we have $\mathcal{Z}^{\rm spl}(\Lambda) \cap \mathcal{Z}^{\rm spl}(\Lambda') = \mathcal{Z}^{\rm spl}(\Lambda'')$ (resp. $\mathcal{Y}^{\rm spl}(\Lambda^\sharp) \cap \mathcal{Y}^{\rm spl}(\Lambda^{\prime \sharp}) = \mathcal{Y}^{\rm spl}(\Lambda^{\prime \prime \sharp})$).

    \item[(ii)] The intersection $\mathcal{Z}^{\rm spl}(\Lambda_0) \cap \mathcal{Z}^{\rm spl}(\Lambda_0')$ (or $\mathcal{Y}^{\rm spl}(\Lambda_0^\sharp) \cap \mathcal{Y}^{\rm spl}(\Lambda_0^{\prime \sharp})$) is always empty if $\Lambda_0 \ne \Lambda'_0$.
    
    \item[(iii)] The intersection $\mathcal{Z}^{\rm spl}(\Lambda) \cap \mathcal{Z}^{\rm spl}(\Lambda_0)$ (resp.  $\mathcal{Y}^{\rm spl}(\Lambda^\sharp) \cap \mathcal{Y}^{\rm spl}(\Lambda_0^\sharp)$) is non-empty if and only if $\Lambda \subset \Lambda_0$ (resp. $\Lambda_0 \subset \Lambda$), in which case $\mathcal{Z}^{\rm spl}(\Lambda) \cap \mathcal{Z}^{\rm spl}(\Lambda_0)$ (resp.  $\mathcal{Y}^{\rm spl}(\Lambda^\sharp) \cap \mathcal{Y}^{\rm spl}(\Lambda_0^\sharp)$) is isomorphic to $\mathbb{P}^{h+t-1}_{\bar{k}}$ (resp. $\mathbb{P}^{h-t-1}_{\bar{k}}$).
    \item[(iv)] The BT-strata $\mathcal{Z}^{\mathrm{spl}}(\Lambda_0)$ and $\mathcal{Y}^{\mathrm{spl}}(\Lambda_0^\sharp)$ are each isomorphic to the projective space $\mathbb{P}^{n-1}_{\bar{k}}$.
   \end{itemize}
   \end{enumerate}
\end{Theorem}

Note that there are no $\calY^\spl$-strata in the BT stratification when $h = 0$, and in this case our results recover those of \cite{HLSY}. We also highlight that, by definition, $\mathcal{Z}^{\rm spl}(\Lambda_1)$ and $\mathcal{Y}^{\rm spl}(\Lambda_2^\sharp)$ are closed subschemes of the special fiber of $\calN_n^\spl$, and we show in Corollary \ref{Reducedness} that these subschemes are reduced. In \cite{HLSY}, the authors also show that the $\calZ^\spl$-strata are reduced for $h = 0$. However, our method is different, more uniform, and applies to any vertex stabilizer level. To prove reducedness, in \S \ref{LPBT}, we show that these BT strata are \'etale locally isomorphic to certain simpler schemes—the strata splitting models $\M_n^{\spl, [2h]}(2t)$ which are closed subschemes of the special fiber of $\M_n^{\spl, [2h]}$ and are introduced in \S \ref{StrataSplModles}. In particular, we construct a local model diagram  

 \begin{equation}
\begin{tikzcd}
&\tilde{\mathcal{Z}}^{\rm spl}(\Lambda_1)\arrow[dl, "\psi_1"']\arrow[dr, "\psi_2"]  & \\
\mathcal{Z}^{\rm spl}(\Lambda_1)  &&  \M^{\spl, [2h]}_n(2t)
\end{tikzcd}
\end{equation}
where $\M^{\spl, [2h]}_n(2t)$ is the strata splitting model with $t> h$.  The morphisms $\psi_1$ and $\psi_2$ are smooth of the same dimension. Similarly, we have a local model diagram for $\mathcal{Y}^{\rm spl}(\Lambda_2^\sharp)$ and $\M^{\spl, [2h]}_n(2t)$ where $t< h$. 

Therefore, to obtain certain nice local properties for the BT-strata, it is enough to study $\M^{\spl, [2h]}_n(2t)$. Similar to the splitting models associated to Shimura varieties, we explicitly calculate an open affine covering $\cup~ \calU_{i_0}$ of $\M^{\spl, [2h]}_n(2t)$. Each affine neighborhood is isomorphic to 
\begin{equation}
\Spec \frac{k[X, Y, Z]}{\left(
{\rm rk}\!\left(
\left[
\begin{array}{c}
X \\ \hline
Y
\end{array}
\right]
\right) - 1,\ 
\wedge^2([Y \mid Z])
\right)}
\end{equation}
for $t> h$. Here, $X$ is a matrix of size $2h\times 1$, and $Y, Z$ are matrices of size $(t-h)\times 1$. The rank condition is expressed by imposing that a certain $i_0$-th entry of the matrix $[X^t \mid Y^t]$ is a unit. When $t<h$, the open affine chart $\calU_{i_0}$  is isomorphic to $\mathbb{A}_k^{n-h-t-1}$ (see Propositions \ref{prop X-strata} and \ref{prop Y-strata} for more details). Studying these affine schemes, in \S \ref{StrataSplModles}, we deduce that: 


\begin{Theorem}
The strata splitting model $\M^{\spl, [2h]}_n(2t)$ is 
normal and Cohen-Macaulay. Moreover, 

(1). For $t> h$, the strata splitting model $\M^{\spl, [2h]}_n(2t)$ has dimension $t+h$.

(2). For $t< h$, excluding the case where $n$ is even and $h=\frac n2$ ($\pi$-modular case), the strata splitting model $\M^{\spl, [2h]}_n(2t)$ is smooth of dimension $n-t-h-1$.  

(3). For $t_2<h<t_1$, the intersection of strata splitting models $\M^{\spl, [2h]}_n(2t_1)\cap \M^{\spl, [2h]}_n(2t_2)$ is smooth of dimension $t_1-t_2-1$.
\end{Theorem}


For a discussion of the $\pi$-modular case, see Remark \ref{Excludemodular}. 
To study the global structure of the $\calZ^\spl$-strata and $\calY^\spl$-strata we establish a scheme-theoretic relationship between BT strata and (modified) Deligne-Lusztig varieties. This is carried out in \S \ref{GlobalPropertiesSection}. Also, due to the extensive notation we omit the discussion of the intersection of $\calZ^\spl$-strata and $\calY^\spl$-strata and we refer to \S \ref{IntZY}.

To give some more details, we first assume that the vertex lattice $\Lambda_1$ is of type $2t$ with $ t>h$ and consider the $k$-vector space $V_{\Lambda_1} = \Lambda_1^{\sharp}/\Lambda_1$ of dimension $2t$ with induced symplectic form $ \langle \,, \, \rangle  $.  Denote by $\Phi$ its Frobenius endomorphism. Let \( \mathrm{Gr}(i, V_{\Lambda_1}) \) be the Grassmannian variety parametrizing rank \( i \) locally direct summands of $ V_{\Lambda_1}$ and let $ \mathrm{SGr}(i, V_{\Lambda_1})$ be the subvariety of $ \mathrm{Gr}(i, V_{\Lambda_1})$ given by $\mathrm{SGr}(i, V_{\Lambda_1}) = \left\{ U \in \mathrm{Gr}(i, V_{\Lambda_1}) \ \middle| \  \langle U , U \rangle =0 \right\}.$ Consider the subvariety  $S'_{\Lambda_1}$ to be the subvariety of $\mathrm{SGr}(t-h, V_{\Lambda_1}) \times \mathrm{Gr}(t+h-1, V_{\Lambda_1})$ whose $\bar{k}$-points are 
\[
S'_{\Lambda_1}(\bar{k}) = \left\{ (U, U') \in \left(\mathrm{SGr}(t-h, V_{\Lambda_1}) \times \mathrm{Gr}(t+h-1, V_{\Lambda_1})\right)(\bar{k}) \;\middle|\; U' \subset U^{\sharp} \cap \Phi(U^{\sharp}) \right\}.
\]
Here $U^{\sharp}$ is the dual of $U$ with respect to the symplectic form $ \langle \,, \, \rangle  $ of $V_{\Lambda_1}$. Then the variety $S'_{\Lambda_1}$ is a projective subvariety of $\mathrm{SGr}(t-h, V_{\Lambda_1}) \times \mathrm{Gr}(t+h-1, V_{\Lambda_1})$. We prove that
\begin{Theorem}
The projective variety $ S'_{\Lambda_1}$ is irreducible and of dimension $t+h$, and is isomorphic to $\mathcal{Z}^{\rm spl}(\Lambda_1)$.     
\end{Theorem}
Next, assume that the vertex lattice $\Lambda_2$ is of type $2t <2h$ and consider the $k$-vector space $V_{\Lambda_2^{\sharp}} = \Lambda_2 /\pi \Lambda^{\sharp}_2$ with induced orthogonal form $ ( \,, \, )  $. Let \( \mathrm{Gr}(i, V_{\Lambda^{\sharp}_2}) \) be the Grassmannian variety and let $ \mathrm{OGr}(i, V_{\Lambda^{\sharp}_2})$ be the subvariety of $ \mathrm{Gr}(i, V_{\Lambda_2^{\sharp}})$ given by $\mathrm{OGr}(i, V_{\Lambda^{\sharp}_2}) = \left\{ U \in \mathrm{Gr}(i, V_{\Lambda_2^{\sharp}}) \ \middle| \  ( U , U ) =0 \right\}.$ Consider the subvariety  $R'_{\Lambda_2^{\sharp}}$ to be the subvariety of $\mathrm{OGr}(h-t, V_{\Lambda^{\sharp}_2}) \times \mathrm{OGr}(h-t-1, V_{\Lambda^{\sharp}_2})$ whose $\bar{k}$-points are specified by
\[
R'_{\Lambda^{\sharp}_2}(\bar{k}) = \left\{ (U, U') \in \left(\mathrm{OGr}(h-t, V_{\Lambda^{\sharp}_2}) \times \mathrm{OGr}(h-t-1, V_{\Lambda^{\sharp}_2})\right)(\bar{k}) \;\middle|\; U' \subset U \cap \Phi(U) \right\}.
\]
\begin{Theorem}
The projective variety $R'_{\Lambda^{\sharp}_2}$ is irreducible and smooth of dimension $n-t-h-1$, and is isomorphic to $\calY^\spl(\Lambda_2^\sharp)$. 
\end{Theorem}

\smallskip

{\bf Acknowledgements:} We thank Y. Luo for his valuable comments and corrections on a preliminary version of this article. I.Z. was supported by Germany's Excellence Strategy EXC~2044--390685587 ``Mathematics M\"unster: Dynamics--Geometry--Structure'' and by the CRC~1442 ``Geometry: Deformations and Rigidity'' of the DFG.

\section{Rapoport-Zink spaces}
In this section, we present the definition and basic properties of certain ramified unitary Rapoport–Zink (RZ) spaces, with level structure given by the stabilizer of a vertex lattice. Since some of these spaces have already appeared in the literature, our discussion will be brief and accompanied by the relevant references.

\subsection{Preliminaries}\label{sec2.1}
Let $F_0$ be a finite extension of $\mathbb{Q}_p$, where $p$ is an odd prime, with residue field $k = \mathbb{F}_q$. Let $\bar{k}$ be a fixed algebraic closure of $k$ and $F$ a ramified quadratic extension of $F_0$. Denote by $a \mapsto \bar{a}$ the (nontrivial) Galois involution of $F/F_0$ and let $\pi$ be a uniformizer of $F$ such that $\bar{\pi} = - \pi$. Let $ \pi_0 = \pi^2$, a uniformizer of $F_0$. Denote by $\br F$ the completion of the maximal unramified extension of $F$ and let $O_F$, $O_{\br F}$ be the ring of integers of $F$, $\br F$ respectively. Denote by $\mathrm{Nilp} \, O_{\br F}$ the category of $O_{\br F}$-schemes $S$ such that $\pi$ is locally nilpotent on $S$ and for such an $S$ denote its special fiber $S \times_{\mathrm{Spf}\,  O_{\br F}} \Spec \bar{k}$  by $\bar{S}$. Let $\sigma \in \mathrm{Gal}(\br{F}_0/F_0)$ be the Frobenius element. We fix an injection of rings $i_0 : O_{F_0} \rightarrow O_{\br{F}_0}$ and an injection $i : O_F \rightarrow O_{\br{F}}$ extending $i_0$. Denote by $ \bar{i} : O_F \rightarrow O_{\br{F}}$ the map $ a \mapsto i(\bar{a})$.

A 
strict $O_{F_0}$-module over $S$, where $S$ is an $O_{F_0}$-scheme, 
is a pair $(X,\iota)$ where $X$ is a $p$-divisible group over $S$ and $\iota : O_{F_0} \longrightarrow \operatorname{End}(X)$ is an action such that $O_{F_0}$ acts on $\operatorname{Lie}(X)$ via the structure morphism $O_{F_0} \to \mathcal{O}_S$. Such an $O_{F_0}$-module is called \emph{formal} if the underlying $p$-divisible group $X$ is formal. By Zink-Lau’s theory, which is generalized in \cite{ACZ}, there is an equivalence of categories between the strict formal $O_{F_0}$-modules over $S$ and nilpotent $O_{F_0}$-displays over $S$ (see also \cite[\S 3.1]{HLS2} and \cite[\S 5]{LMZ} for more details).
 To any strict formal $O_{F_0}$-module, there is an associated crystal $\mathbb{D}_X$ on the category of $O_{F_0}$-pd-thickenings. We define the (covariant relative) de~Rham realization as $D(X) := \mathbb{D}_X(S)$ and by the (relative) Grothendieck--Messing theory we obtain a short exact sequence of $\mathcal{O}_S$-modules:
\[
0 \longrightarrow \operatorname{Fil}(X) \longrightarrow D(X) \longrightarrow \operatorname{Lie}(X) \longrightarrow 0,
\]
where $\operatorname{Fil}(X) \subset D(X)$ is the Hodge filtration. (See \cite[\S 3.1]{HLS2} for a more comprehensive treatment.)

Next, we restrict to the case where $X = (X, \iota)$ is biformal; see \cite[Definition~11.9]{Mih} for the definition. For a biformal strict $O_{F_0}$-module $X$, we can define the (relative) dual $X^\vee$ of $X$, and hence the (relative) polarization and the (relative) height. From the definition, it follows that there is a perfect pairing 
\begin{equation}\label{perf.pair.}
D(X) \times D(X^\vee) \to \mathcal{O}_S
\end{equation}
such that $\operatorname{Fil}(X) \subset D(X)$ and $\operatorname{Fil}(X^\vee) \subset D(X^\vee)$ are orthogonal complements of each other and there are two induced perfect pairings 
\[
\operatorname{Fil} (X) \times \operatorname{Lie}(X^\vee)\to \mathcal{O}_S \text{  and  } \operatorname{Fil} (X^\vee) \times \operatorname{Lie}(X)\to \mathcal{O}_S.
\]

When $S = \operatorname{Spec} R$ is perfect, the nilpotent $O_{F_0}$-display is equivalent to the relative Dieudonn\'e module $M(X)$ over $W_{O_{F_0}}(R)$, equipped with a $\sigma$-linear operator $F$ and a $\sigma^{-1}$-linear operator $V$, such that $FV = VF = \pi \cdot \mathrm{id}$. (Here, $W_{O_{F_0}}(R)$ is the ring of ramified Witt vectors and we refer the reader to \cite[\S 3.1]{HLS2} for more details.)
 
\begin{Definition}\label{def 21}
\rm{
 Let $h, n$ be integers with $0 \leq h \leq \lfloor \frac n2 \rfloor$. For any $S \in \mathrm{Nilp}\, O_{\br F}$, a \textit{hermitian} $O_F$-\textit{module} of rank $n$ and type $2h$ (with signature $(n-1, 1)$) over $S$ is a triple $(X, \iota, \lambda)$ satisfying:
\begin{enumerate}
  \item $X$ is a strict biformal $O_{F_0}$-module over $S$ of height $2n$ and dimension $n$.
  \item $\iota : O_F \to \operatorname{End}(X)$ is an action of $O_F$ on $X$ extending the $O_{F_0}$-action.
  \item $\lambda$ is a (relative) polarization of $X$ that is $O_F/O_{F_0}$-semilinear in the sense that the Rosati involution $\operatorname{Ros}_\lambda$ induces the non-trivial involution $\sigma \in \operatorname{Gal}(F/F_0)$ on $\iota : \mathcal{O}_F \to \operatorname{End}(X)$.
  \item We require that $\ker[\lambda] \subset X[\iota(\pi)]$ and has order $q^{2h}$.
\end{enumerate}
}
\end{Definition}
From (4) above, we deduce that there exists a unique isogeny $\lambda^\vee : X^\vee \to X$ such that $\lambda \circ \lambda^\vee = \iota(\pi)$ and $\lambda^\vee \circ \lambda = \iota(\pi)$.

\subsection{Unitary RZ spaces}\label{sec2.2} We fix a supersingular hermitian $O_F$-module $(\mathbb{X},\iota_{\mathbb{X}},\lambda_{\mathbb{X}})$ over $\Spec \bar{k}$ of rank $n$ and type $2h$ (with signature $(n-1, 1)$) which we call the framing object; supersingular means that the rational  Dieudonn\'e module $N = M(\mathbb{X})[1/\pi_0]$ has all relative slopes $\frac{1}{2}$. (We refer to \cite[\S 5]{LRZ} for the existence of these framing objects.) Now, we are ready to define the following RZ spaces which are \textit{relative} in the sense of \cite{Mih}. 

\begin{Definition}
{\rm 
\begin{enumerate}
    \item The wedge RZ space $\mathcal{N}_n^{\wedge}$ is the set-valued functor on $\mathrm{Nilp} \, O_{\breve{F}}$ which associates to $S \in \mathrm{Nilp} \, O_{\breve{F}}$ the set of isomorphism classes of quadruples $(X, \iota, \lambda, \rho)$ which satisfy
\begin{enumerate}
    \item  $(X, \iota, \lambda)$ is a hermitian  $O_F$-module over S of dimension $n$ and type $2h$.
    \item $\rho : X \times_S \bar{S} \to \mathbb{X} \times_{\bar{k}} \bar{S}$ is an $O_F$-linear quasi-isogeny of height $0$ over the special fiber $\bar{S} = S \times_{\mathrm{Spf}\,  O_{\br F}} \Spec \bar{k}$ such that $\rho^*(\lambda_{\mathbb{X}, \bar{S}}) = \lambda_{\bar{S}}$.
    \item The action of $O_F$ on $\operatorname{Fil}(X)$ induced by $\iota : O_F \to \operatorname{End}(X)$ satisfies:
 \begin{itemize}
  \item \textit{(Kottwitz condition)}: 
  $ \operatorname{char}(\iota(\pi) \mid \operatorname{Fil}(X)) = (T - \pi)(T + \pi)^{n-1}.$   \item \textit{(Wedge condition)}: 
 $\wedge^2(\iota(\pi) - \pi \mid \operatorname{Fil}(X)) = 0, \,  \wedge^n(\iota(\pi) + \pi \mid \operatorname{Fil}(X)) = 0.$ 
 \end{itemize}
 \item \textit{(Spin condition)} When $n$ is even and $2h = n$, we ask that $\iota(\pi) - \pi$ is non-vanishing on $\mathrm{Fil}(X)$.
\end{enumerate}

\item The RZ space $\mathcal{N}_n^{\rm loc}$ is defined as the closed formal subscheme of $\mathcal{N}_n^{\wedge}$ cut out by the ideal sheaf $\mathcal{O}_{\mathcal{N}_n^{\wedge}}[\pi_0^\infty] \subset \mathcal{O}_{\mathcal{N}_n^{\wedge}}$. This is the maximal flat closed formal subscheme of $\mathcal{N}_n^{\wedge}$.
\end{enumerate}
}\end{Definition}

The RZ spaces $\mathcal{N}_n^{\wedge}$ and $\mathcal{N}_n^{\rm loc}$ are representable by formal schemes locally of finite type over $\mathrm{Spf} \,O_{\br F}$ and both spaces have relative dimension $n - 1$ (see \cite[\S 3.3]{HLS2}). The closed formal subscheme $\mathcal{N}_n^{\rm loc}$ is flat and has the same underlying topological space with $\mathcal{N}_n^{\wedge}$, i.e. these spaces share identical reduced loci. These assertions can be easily seen by using the local model diagram and passing to the corresponding local models ${\rm M}_n^{\wedge}$ and ${\rm M}_n^{\rm loc}$ (see \cite[Proposition 3.4]{HLS2}). From \cite[\S 3.3]{HLS2}, we also see that $\mathcal{N}_n^{\rm loc}$ is a linear modification of $\mathcal{N}_n^{\wedge}$ in the sense of \cite[\S 2]{P}.
\begin{Definition}\label{SplDef}
\rm{
 The splitting RZ space $\mathcal{N}_n^{\rm spl}$ is the set-valued functor on $\mathrm{Nilp} \, O_{\breve{F}}$ which associates to $S \in \mathrm{Nilp} \, O_{\breve{F}}$ the set of isomorphism classes of quintuples $(X, \iota, \lambda, \rho,\operatorname{Fil}^0(X))$ which satisfy
\begin{enumerate}
 \item  $(X, \iota, \lambda)$ is a hermitian  $O_F$-module over $S$ of dimension $n$ and type $2h$.
    \item $\rho : X \times_S \bar{S} \to \mathbb{X} \times_{\bar{k}} \bar{S}$ is an $O_F$-linear quasi-isogeny of height $0$ over the special fiber $\bar{S}$ such that $\rho^*(\lambda_{\mathbb{X}, \bar{S}}) = \lambda_{\bar{S}}$.
    \item $\operatorname{Fil}^0(X)$ is locally a $\mathcal{O}_S$-direct summand of the Hodge filtration $\operatorname{Fil}(X) \subset D(X)$ of rank one that satisfies the \textit{splitting conditions} 
    \[  (\iota(\pi)+\pi) (\operatorname{Fil}(X)) \subset \operatorname{Fil}^0(X) \quad \text{and} \quad (\iota(\pi)-\pi) (\operatorname{Fil}^0(X))=0.
    \]  
    \item \textit{(Spin condition)} When $n$ is even and $2h = n$, we ask that $\iota(\pi) - \pi$ is non-vanishing on $\mathrm{Fil}(X)$.
\end{enumerate}
}
\end{Definition}
The RZ space $\mathcal{N}_n^{\rm spl}$ is representable by a flat formal scheme of relative dimension $n - 1$ over $\operatorname{Spf} O_{\breve F}$. The representability follows from the general results of \cite{RZbook}. Also, using the same arguments as in \cite[\S 5.2]{HLS1} we obtain a local model diagram  which connects $\mathcal{N}_n^{\rm spl}$ with the splitting model ${\rm M}_n^{\rm spl}$ defined in \cite{ZacZhao2}. By Proposition \ref{SplitFlat} we have that ${\rm M}_n^{\rm spl}$, and so $\mathcal{N}_n^{\rm spl}$, is normal and flat. In particular, as above, we have that $\mathcal{N}_n^{\rm spl}$ is a linear modification of  $\mathcal{N}_n^{\wedge}$ (see \cite[\S 8]{ZacZhao2}). As in loc. cit., there is a forgetful map 
\[
\phi: \mathcal{N}_n^{\rm spl} \rightarrow \mathcal{N}_n^{\wedge}
\]
defined by $(X, \iota, \lambda, \rho, \operatorname{Fil}^0(X)) \mapsto (X, \iota, \lambda, \rho) $ which factors through $\mathcal{N}_n^{\rm loc} \subset \mathcal{N}_n^{\wedge} $ because of flatness (see also \cite[\S 1.12.2]{LRZ}). 

\begin{Remarks}\label{Excludemodular}{\rm
\begin{enumerate}
    \item We note that in our definition of $\mathcal{N}_n^{\rm spl}$, which imitates the definition of ${\rm M}_n^{\rm spl}$, we add only one subspace ($\operatorname{Fil}^0(X)$) instead of adding two subspaces ($\operatorname{Fil}^0(X) $  and $\operatorname{Fil}^0(X^{\wedge}) $) as expected from \cite[Definition 14.1]{PR2}; this variation is necessary to get a flat model as observed in \cite{ZacZhao2}.
    \item In the $\pi$-modular case, i.e. $n$ is even and $2h =n$, the splitting model ${\rm M}_n^{\rm spl}$ is smooth and equals the local model $ {\rm M}_n^{\rm loc}$ (see \cite[Remark 5.12]{ZacZhao1}). 
    We exclude this case, as the basic locus of the corresponding RZ-space has already been described in \cite{Wu} and no new phenomena arise from our calculations. 
\end{enumerate}
}\end{Remarks}

\begin{Remark}\label{DualSpace}
{\rm
Denote by $\mathcal{F} \subset  \operatorname{Lie}(X^\vee)$ the perpendicular complement of $ \operatorname{Fil}^0(X)$ under the perfect pairing (\ref{perf.pair.}) which determine each other. Then, as in \cite[\S 3.2]{HLSY}, condition (3) of Definition \ref{SplDef} is equivalent to: $\mathcal{F}$ is a local direct summand of $\operatorname{Lie} X^\vee$ of rank $n - 1$ as an $\mathcal{O}_S$-module such that $\mathcal{O}_F$ acts on $\mathcal{F}$ via $O_F \xrightarrow{i} O_{\br{F}} \to \mathcal{O}_S$ and acts on $\operatorname{Lie} X^\vee / \mathcal{F}$ via $O_F \xrightarrow{\bar{i}} O_{\br{F}} \to \mathcal{O}_S$.
}\end{Remark}

Recall that we denote by $N = M(\mathbb{X})[1/\pi_0]$ the rational Dieudonn\'e module of the framing object which is a $2n$-dimensional $\breve{F}_0$-vector space equipped with a $\sigma$-linear operator $F$ and a $\sigma^{-1}$-linear operator $V$. 
The $O_F$-action $\iota_{\mathbb{X}} : O_F \to \operatorname{End}(\mathbb{X})$ induces on $N$ an $O_F$-action commuting with $F$ and $V$. We still denote this induced action by $\iota_{\mathbb{X}}$ and denote $\iota_{\mathbb{X}}(\pi)$ by $\Pi$.

The polarization of $\mathbb{X}$ induces a skew-symmetric $\breve{F}_0$-bilinear form $\langle \cdot, \cdot \rangle$ on $N$ satisfying
\[
\langle Fx, y \rangle = \langle x, Vy \rangle^\sigma, \qquad \langle \iota(a)x, y \rangle = \langle x, \iota(\bar{a})y \rangle,
\]
for any $x, y \in N$, $a \in O_F$. Also, $N$ is an $n$-dimensional $\breve{F}$-vector space equipped with the $F/\breve{F}_0$-hermitian form $h(\cdot, \cdot)$ defined by
\[
h(x, y) := \delta \left( \langle \Pi x, y \rangle + \pi \langle x, y \rangle \right),
\]
where $\delta$ is a fixed element in $\breve{F}_0^\times$ satisfying $\sigma(\delta) = -\delta$. The bilinear form $\langle \cdot, \cdot \rangle$ can be recovered from $h(\cdot, \cdot)$ via the relation:
\[
\langle x, y \rangle = \frac{1}{2\delta} \operatorname{Tr}_{F/\breve{F}_0} \left( \pi^{-1} h(x, y) \right).
\]
For a lattice $\Lambda \subset N$ we denote by $ \Lambda^\sharp = \{x \in N \,|\, h(x,\Lambda) \in O_F \}$ its hermitian dual. Let $\tau := \Pi V^{-1}$ and define $C := N^{\tau = 1}$ (the set of $\tau$-fixed points in $N$). Then $C$ is an $F$-vector space of dimension $n$ and we have
\[
N = C \otimes_{F_0} \breve{F}_0.
\]
The $F/F_0$-hermitian form $h(\cdot, \cdot)$ restricts to $C$ and we continue to use the same notation for the restricted form. From now on, we write $\pi$ instead of $\Pi$ for the action on $C$. We call $\Lambda\subset C$ a vertex lattice if it satisfies 
\begin{equation}
\pi \Lambda^{\sharp} \subset \Lambda \subset  \Lambda^{\sharp}.   
\end{equation}
We denote by $t(\Lambda):=\dim(\Lambda^\sharp/\Lambda)$ the type of a vertex lattice, which is an even integer (see \cite[Lemma 3.2]{RTW}). Also, set $\breve{\Lambda} = \Lambda \otimes_{O_F} O_{\breve{F}}$. 

\begin{Proposition}\label{DLoc}
Let $ \kappa$ be a perfect field over $ \bar{k}$. There is a bijection between $\mathcal{N}_n^{\rm loc}(\kappa)$ and the set of $W_{O_{F_0}}(\kappa)$-lattices
{\small
\[
\left\{ 
M \subset N \otimes W_{O_{F_0}}(\kappa) \,\middle|\, 
\Pi M^\sharp \subset M \mathrel{\overset{\scriptstyle 2h}{\subset}} M^\sharp,\ 
\Pi M \subset \tau^{-1}(M) \subset \Pi^{-1} M,\ 
M \mathrel{\overset{\scriptstyle \leq 1}{\subset}} M + \tau(M) 
\right\}.
\]
}
\end{Proposition}
\begin{proof}
See \cite[Proposition 3.5]{HLS2}. 
\end{proof}

\begin{Proposition}\label{DSpl}
Let $ \kappa$ be a perfect field over $ \bar{k}$. There is a bijection between $\mathcal{N}_n^{\rm spl}(\kappa)$ and the set of pairs of $W_{O_{F_0}}(\kappa)$-lattices $(M,M')$ in $N \otimes W_{O_{F_0}}(\kappa)$ satisfying
\[
\begin{array}{l}
\Pi M^\sharp \subset M \mathrel{\overset{\scriptscriptstyle 2h}{\subset}} M^\sharp,\quad 
\Pi M \subset \tau^{-1}(M) \subset \Pi^{-1} M,   \\
V M^\sharp \subset M' \subset \tau^{-1}(M^\sharp) \cap M^\sharp,\quad
\operatorname{length}(M^\sharp / M') = 1.     
\end{array}
\]

\end{Proposition}
\begin{proof}
Let $(X, \iota, \lambda, \rho, \mathcal{F}) \in \mathcal{N}_n^{\rm spl}(\kappa) $ and let $M(X)$ be the $O_{F_0}$-relative Dieudonné module of $X$. Define $M = \rho(M(X)) \subset N  \otimes W_{O_{F_0}}(\kappa)$ and $M' = \rho(\operatorname{Pr}^{-1}(\mathcal{F})) \subset N  \otimes W_{O_{F_0}}(\kappa)$, where $\operatorname{Pr} : M(X^\vee) \to \operatorname{Lie} X^\vee = M(X^\vee)/V M(X^\vee)$ is the natural quotient map.

As in the proof of \cite[Proposition 3.5]{HLS2} the relation $\pi M^\sharp \subset M \mathrel{\overset{\scriptscriptstyle 2h}{\subset}} M^\sharp$ comes from the polarization $\lambda$ and the relation $\Pi M \subset \tau^{-1}(M) \subset \Pi^{-1} M$ is equivalent to $ \pi_0 M \subset VM \subset M $. (Note that the Hodge filtration $\operatorname{Fil}(X) \subset D(X)$ can be identified with $ VM / \pi_0M \subset M / \pi_0M$.) The conditions $V M^\sharp \subset M' \subset \tau^{-1}(M^\sharp) \cap M^\sharp$ and $\operatorname{length}(M^\sharp/M') = 1$ are equivalent to
\[V M^\sharp \subset M' \subset M^\sharp, \quad \Pi M' \subset V M^\sharp, \quad \dim_{\kappa}(M^\sharp/M') = 1,
\] 
which are in turn equivalent to 
\[
\mathcal{F} \subset \operatorname{Lie} X^\vee, \quad \dim_{\kappa}(\mathcal{F}) = n - 1. \quad \Pi \cdot \mathcal{F} = \{0\}, \quad \Pi \cdot \operatorname{Lie} X^\vee \subset \mathcal{F} .
\]
Notice that the condition $\Pi \cdot \operatorname{Lie} X^\vee  \subset \mathcal{F}$ is automatic once we know $\dim_{\kappa}(\mathcal{F}) = n - 1$ and $\mathcal{F}$ is stable under the action of $\Pi$ (see also the proof of \cite[Proposition 3.5]{HLSY}). Hence the filtration $\mathcal{F} \subset \operatorname{Lie} X^\vee $ satisfies the splitting conditions. By combining Remark \ref{DualSpace} with the above, we have translated all conditions in the definition of $\mathcal{N}_n^{\rm spl}$ in terms of relative Dieudonné modules.
\end{proof}

\subsection{Bruhat-Tits strata}\label{BT_strata} 
Let $\mathcal{N}_1$ be the height two relative Rapoport-Zink space
with strict $O_F$-action and we fix the framing object $ (\mathbb{Y}, \iota_{\mathbb{Y}}, \lambda_{\mathbb{Y}})$ of dimension one over $\Spec \bar{k}$.  Define 
\begin{equation} \label{eq:Vn-def}
\mathbb{V} := \operatorname{Hom}_{O_F}(\mathbb{Y}, \mathbb{X}) \otimes_{O_F} F.
\end{equation}
The vector space $\mathbb{V} $ is equipped with a hermitian form $(\, ,\, )_{\mathbb{V}} $ such that for any $x, y \in \mathbb{V} $,
\begin{equation} \label{eq:hermitian-V}
(x, y)_{\mathbb{V} } = \lambda_{\mathbb{Y}}^{-1} \circ y^\vee \circ \lambda_{\mathbb{X}} \circ x \in \operatorname{End}(\mathbb{Y}) \otimes_{O_F} F \cong F, 
\end{equation}
where $y^\vee$ is the dual quasi-homomorphism of $y$ and the last isomorphism is given by $\iota^{-1}_{\mathbb{Y}}$. The hermitian spaces $(\mathbb{V}, (\, ,\, )_{\mathbb{V}})$ and $(C, h(\,, \,))$ are related by the $F$-linear isomorphism
\begin{equation} \label{eq:b-isom}
b : \mathbb{V} \to C, \quad x \mapsto x(e),
\end{equation}
where $e$ is a generator of the $\tau$-fixed points of the $O_{F_0}$-relative Dieudonné module $M(\mathbb{Y})$; in particular, $\mathbb{V}$ and $C$ are isomorphic as hermitian spaces (see \cite[\S 2.2]{HLSY}). We will sometimes identify $\mathbb{V}$ with $C$.

\quash{\begin{Definition}\label{Def1}
  \rm{
(1) For an $O_F$ lattice $L \subset \mathbb{V}$, define the subfunctor $\mathcal{Z}(L)$ of $ \mathcal{N}_n^{\rm loc}$ such that $\mathcal{N}_n^{\rm loc}(S)$ is the set of isomorphism classes of tuples $ (X, \iota, \lambda, \rho) \in \mathcal{N}_n^{\rm loc}(S)$ such that for any $x \in L \subset \mathbb{V}$ the quasi-homomorphism $ \rho^{-1} \circ x \circ \rho_Y: \mathbb{Y} \times_{\bar{k}} \bar{S} \to X \times_S \bar{S}$ extends to a homomorphism $ Y \to X $.

(2) For an $O_F$ lattice $L \subset \mathbb{V}$, define the subfunctor $\mathcal{Y}(L^{\sharp})$ of $ \mathcal{N}_n^{\rm loc}$ such that $\mathcal{N}_n^{\rm loc}(S)$ is the set of isomorphism classes of tuples $ (X, \iota, \lambda, \rho) \in \mathcal{N}_n^{\rm loc}(S)$ such that for any $x^{^\sharp} \in \Lambda^{\sharp} \subset \mathbb{V}$ the quasi-homomorphism $ \rho^\vee \circ \lambda_{\mathbb{X}} \circ x^{^\sharp} \circ \rho_Y: \mathbb{Y} \times_{\bar{k}} \bar{S} \to X^\vee \times_S \bar{S}$ extends to a homomorphism $ Y \to X^\vee $.
}
\end{Definition}}

By (relative) Dieudonn\'e theory, the lattices $\breve{\Lambda}$ and $\breve{\Lambda}^\sharp$ correspond to the strict $O_{\breve{F}_0}$-modules $X_\Lambda$ and $X_{\Lambda^\sharp}$ over $\bar{k}$, respectively, with quasi-isogenies $\rho_\Lambda : X_\Lambda \to \mathbb{X}$ and $\rho_{\Lambda^\sharp} : X_{\Lambda^\sharp} \to \mathbb{X}$. We define the following two kinds of Bruhat--Tits (BT) strata for the special fiber $\overline{\mathcal{N}}_n^{\rm loc}$  of $\mathcal{N}_n^{\rm loc}$
(see also \cite[\S 3.3]{HLSY} and \cite[Definition 2.2]{HLS2}):


\begin{Definition}\label{Isogenies}
\rm{
Fix an even integer $0\leq 2h\leq n$. Let $L_{\mathcal{Z}}$ denote the set of all vertex lattices in $C$ of type $2t \geq 2h$, and let $L_{\mathcal{Y}}$ denote the set of all vertex lattices in $C$ of type $2t\leq 2h$. 

\begin{enumerate}
  \item For any $\Lambda \in L_{\mathcal{Z}}$, the $\mathcal{Z}$-stratum $\mathcal{Z}^{\rm loc}(\Lambda)$ is the subfunctor of $\overline{\mathcal{N}}_n^{\rm loc}$ that assigns to each $\bar{k}$-scheme $S$ the set of tuples $(X, \iota, \lambda, \rho)$ such that the composition $\rho_{\Lambda,X} := \rho^{-1} \circ (\rho_\Lambda)_S$ is an isogeny.

  \item For any $\Lambda \in L_{\mathcal{Y}}$, the $\mathcal{Y}$-stratum $\mathcal{Y}^{\rm loc}(\Lambda^\sharp)$ is the subfunctor of $\overline{\mathcal{N}}_n^{\rm loc}$ that assigns to each $\bar{k}$-scheme $S$ the set of tuples $(X, \iota, \lambda, \rho)$ such that the composition $\rho_{\Lambda^\sharp, X^\vee} := \rho^\vee \circ \lambda_{\mathbb{X}} \circ \rho_{\Lambda^\sharp}$ is an isogeny, where $\rho_{\Lambda^\sharp} = \rho_\Lambda \circ \lambda_\Lambda^{-1}$.
\end{enumerate}
}
\end{Definition}
By \cite[Lemma 2.10]{RZbook}, $\mathcal{Z}^{\rm loc}(\Lambda)$ and $\mathcal{Y}^{\rm loc}(\Lambda^\sharp)$ are closed formal subschemes of $\overline{\mathcal{N}}_n^{\rm loc}$ (see also \cite[\S 2.1]{HLS2}). Using the same reasoning as in \cite[Lemma 4.2]{VW}, it follows that they are representable by projective schemes over $\bar{k}$. Also, these schemes are reduced (see \cite[Corollary 4.8]{HLS2}) and so they lie in the reduced subscheme $  \calN_{n, {\rm red}}^{\rm loc}$ of $  \overline{\mathcal{N}}_n^{\rm loc}$.
\quash{
\begin{Definition}\label{Defsplcycles}
  \rm{
(1) For an $O_F$ lattice $L \subset \mathbb{V}$, define the closed subfunctor $\mathcal{Z}'(L)$ of $ \mathcal{N}_n^{\rm spl}$ such that $\mathcal{N}_n^{\rm spl}(S)$ is the set of isomorphism classes of tuples $ (X, \iota, \lambda, \rho, \calF) \in \mathcal{N}_n^{\rm spl}(S)$ such that for any $x \in L \subset \mathbb{V}$ the quasi-homomorphism $ \rho^{-1} \circ x \circ \rho_Y: \mathbb{Y} \times_{\bar{k}} \bar{S} \to X \times_S \bar{S}$ extends to a homomorphism $ Y \to X $.

(2) For an $O_F$ lattice $L \subset \mathbb{V}$, define the closed subfunctor $\mathcal{Y}'(L^{\sharp})$ of $ \mathcal{N}_n^{\rm spl}$ such that $\mathcal{N}_n^{\rm spl}(S)$ is the set of isomorphism classes of tuples $ (X, \iota, \lambda, \rho, \calF) \in \mathcal{N}_n^{\rm spl}(S)$ such that for any $x^{^\sharp} \in \Lambda^{\sharp} \subset \mathbb{V}$ the quasi-homomorphism $ \rho^\vee \circ \lambda_{\mathbb{X}} \circ x^{^\sharp} \circ \rho_Y: \mathbb{Y} \times_{\bar{k}} \bar{S} \to X^\vee \times_S \bar{S}$ extends to a homomorphism $ Y \to X^\vee $.
}
\end{Definition}
}

Next, we define the corresponding strata for the special fiber of the splitting RZ-space $ \overline{\mathcal{N}}_n^{\rm spl}$. 
As in \cite[\S 3.2]{HLSY}, for a vertex lattice $\Lambda\subset \mathbb{V}$, we have that for each $\bar{k}$-scheme $S$: 
\begin{enumerate}
    \item $\mathcal{Z}^{\rm spl}(\Lambda) (S) $ is the set of isomorphism classes of tuples $(X, \iota, \lambda, \rho, \mathcal{F}) \in \overline{\mathcal{N}}_n^{\rm spl} (S)$ such that $(X, \iota, \lambda, \rho) \in \mathcal{Z}^{\rm loc}(\Lambda)(S)$ and if $\Lambda$ is of type $2t \neq 2h $, we require in addition that $ x_{*} (\operatorname{Lie}(Y\times S)) \subset \mathcal{F}$ for any $x \in \Lambda$. 
    \item $\mathcal{Y}^{\rm spl}(\Lambda^\sharp) (S)$ is the set of isomorphism classes of tuples $(X, \iota, \lambda, \rho, \mathcal{F}) \in \overline{\mathcal{N}}_n^{\rm spl}(S) $ such that $(X, \iota, \lambda, \rho) \in \mathcal{Y}^{\rm loc}(\Lambda^\sharp)(S)$ and if $\Lambda$ is of type $2t \neq 2h $, we require in addition that $ x^{\sharp}_{*} (\operatorname{Lie}(Y\times S)) \subset \mathcal{F}$ for any $x^{\sharp} \in \Lambda^{\sharp}$. 
\end{enumerate}
Note that in Corollary \ref{Reducedness} we will show that the moduli functors $\mathcal{Z}^{\rm spl}(\Lambda) $ and  $\mathcal{Y}^{\rm spl}(\Lambda^{\sharp}) $ are reduced and thus they lie in the reduced subscheme $\calN_{n, {\rm red}}^\spl$ of $\overline{\mathcal{N}}_n^{\rm spl}$. Using Propositions \ref{DLoc} and \ref{DSpl} we can naturally obtain a lattice-theoretic characterization of BT-strata for $\overline{\mathcal{N}}_n^{\rm spl}$ and $ \overline{\mathcal{N}}_n^{\rm loc}$:
\begin{Proposition}\label{DieudLatt}
Let $ \kappa$ be a perfect field over $\bar{k}$. The $\kappa$-points of the BT-strata can be described as follows: 

(1) Assume $\Lambda\subset C$ is a vertex lattice of type $2t \geq 2h$.
\begin{itemize}
    \item For $t= h$, we have
    \begin{flalign*}
    & \mathcal{Z}^{\rm loc}(\Lambda)(\kappa) = \left\{ (X, \iota, \lambda, \rho) \in \mathcal{N}_n^{\rm loc}(\kappa) \mid \Lambda \otimes W_{O_{F_0}}(\kappa) = M(X)  \right\}, \\  
    & \mathcal{Z}^{\rm spl}(\Lambda)(\kappa) = \left\{ (X, \iota, \lambda, \rho,\mathcal{F}) \in \mathcal{N}^{\rm spl}_{n}(\kappa) \,\middle|\, \Lambda \otimes W_{O_{F_0}}(\kappa) = M(X)  \right\}.
    \end{flalign*}
   \item For $t > h$, we have that $ \mathcal{Z}^{\mathrm{loc}}(\Lambda)(\kappa)$ is equal to
   \[  
  \left\{ (X, \iota, \lambda, \rho) \in \mathcal{N}_n^{\rm loc}(\kappa) \;\middle|\;
  \Lambda \otimes W_{O_{F_0}}(\kappa) \subset M(X) \subset M(X)^\sharp \subset \Lambda^\sharp \otimes W_{O_{F_0}}(\kappa)
  \right\}
   \]
and $\mathcal{Z}^{\mathrm{spl}}(\Lambda)(\kappa) $ is equal to 
\[
\left\{ (X, \iota, \lambda, \rho, \mathcal{F}) \in \mathcal{N}^{\mathrm{spl}}_{n}(\kappa) \;\middle|\;
  \begin{aligned}
    &(X, \iota, \lambda, \rho) \in \mathcal{Z}^{\mathrm{loc}}(\Lambda)(\kappa), \\
    &\Lambda \otimes W_{O_{F_0}}(\kappa) \subset M'(X) \subset M(X)^\sharp
  \end{aligned}
  \right\}.
\]
\end{itemize}

(2) Assume $\Lambda\subset C$ is a vertex lattice of type $2t \leq 2h$.
\begin{itemize}
    \item For $t=h$, we have
    \begin{flalign*}
   &  \mathcal{Y}^{\rm loc}(\Lambda^{\sharp})(\kappa) = \left\{ (X, \iota, \lambda, \rho) \in \mathcal{N}_n^{\rm loc}(\kappa) \mid \Lambda \otimes W_{O_{F_0}}(\kappa) = M(X)  \right\}, \\  
  &   \mathcal{Y}^{\rm spl}(\Lambda^{\sharp})(\kappa) = \left\{ (X, \iota, \lambda, \rho,\mathcal{F}) \in \mathcal{N}^{\rm spl}_{n}(\kappa) \mid \Lambda \otimes W_{O_{F_0}}(\kappa) = M(X)  \right\}.
    \end{flalign*}
\item For $t < h$, we have that $ \mathcal{Y}^{\mathrm{loc}}(\Lambda^\sharp)(\kappa)$ is equal to
\[
\left\{ (X, \iota, \lambda, \rho) \in \mathcal{N}_n^{\rm loc}(\kappa) \;\middle|\;
  \pi \Lambda^\sharp \otimes W_{O_{F_0}}(\kappa) \subset \pi M(X)^\sharp \subset M(X) \subset \Lambda \otimes W_{O_{F_0}}(\kappa)
  \right\}
\]
and $ \mathcal{Y}^{\mathrm{spl}}(\Lambda^\sharp)(\kappa)$ is equal to
\[
\left\{ (X, \iota, \lambda, \rho, \mathcal{F}) \in \mathcal{N}^{\mathrm{spl}}_n(\kappa) \;\middle|\;
  \begin{aligned}
    &(X, \iota, \lambda, \rho) \in \mathcal{Y}^{\mathrm{loc}}(\Lambda^\sharp)(\kappa), \\
    &\Lambda^\sharp \otimes W_{O_{F_0}}(\kappa) \subset M'(X) \subset M(X)^\sharp
  \end{aligned}
  \right\}.
\]
\end{itemize}
\qed
\end{Proposition}
\begin{Corollary}\label{t0}
Let $ \kappa$ be a perfect field over $\bar{k}$ and let $\Lambda\subset C$ be a vertex lattice of type $2h$. Then
\begin{enumerate}
    \item $ \mathcal{Z}^{\rm loc}(\Lambda)(\kappa) =  \mathcal{Y}^{\rm loc}(\Lambda^{\sharp})(\kappa)$ and as sets contain a discrete point in the {\rm RZ} space called the \emph{worst point}.
    \item Both strata $ \mathcal{Z}^{\rm spl}(\Lambda)(\kappa)$ and $ \mathcal{Y}^{\rm spl}(\Lambda^\sharp)(\kappa)$ are isomorphic to $\mathbb{P}^{n-1}_{\kappa}$.
\end{enumerate}
\end{Corollary}
\begin{proof}
Both claims follow from Proposition \ref{DieudLatt}. More precisely, for the first claim we have $ \mathcal{Z}^{\rm loc}(\Lambda)(\kappa) =  \mathcal{Y}^{\rm loc}(\Lambda^{\sharp})(\kappa) =\{\Lambda \otimes W_{O_{F_0 }}(\kappa) \}$. 
For the second claim, since $ M =\Lambda \otimes W_{O_{F_0}} (\kappa)$ we can easily see from the above constructions, see Proposition \ref{DSpl} and its proof, that $\mathcal{F}$ can be any rank $n - 1$ locally free $\kappa$-module on $\operatorname{Lie} X^\vee$. 
\end{proof}
Note that under the local model diagram, the points $ M = \Lambda \otimes W_{O_{F_0 }}(\kappa) \in \mathcal{Z}^{\rm loc}(\Lambda)(\kappa)$, where $\Lambda$ is a vertex lattice of type $2h$, correspond to the worst point of the local model ${\rm M}_n^{\rm loc}$. This justifies the terminology worst point introduced in the above corollary (see also \cite[\S 5.3]{LRZ}).


\section{Strata splitting models}\label{StrataSplModles}

In this section, we will introduce the strata splitting models which are defined by purely linear algebraic data. As we will see in Section \ref{LPBT}, these models are \'{e}tale locally isomorphic to the Bruhat-Tits strata $\calZ^\spl(\Lambda)$ and $\calY^\spl(\Lambda^\sharp)$. Therefore, it is enough to study these easier models to obtain several geometric ``local" properties for the corresponding BT-strata. To define these, we first introduce the strata local models for the RZ space $\mathcal{N}_n^{\rm loc}$. We will then see that, just as the splitting RZ spaces $\mathcal{N}_n^{\rm spl}$ are linear modifications of $\mathcal{N}_n^{\rm loc}$, so too are the strata splitting models of the corresponding strata local models.

\subsection{Review of strata local models}

In this subsection, we will briefly review the strata local models of $\calZ^\loc(\Lambda)$ and $\calY^\loc(\Lambda^\sharp)$. These models are reduced, normal, Cohen-Macaulay with dimensions depending on the type of lattice $\Lambda$ and the vertex stabilizer level of the RZ space $\mathcal{N}_n^{\rm loc}$ (see Theorem \ref{StrataLocalModel}). We refer to \cite[\S 4]{HLS2} for more details on strata local models. 

Recall that $F/F_0$ is a ramified quadratic extension and $\pi \in F$ (resp. $\pi_0$) is a uniformizer of $F$ (resp. $F_0$) with $\pi^2 = \pi_0$. Let $k$ be the perfect residue field of characteristic $\neq 2$.  Consider the $F$-vector space $F^n$ of dimension $n > 3$ and let 
\[
h: F^n \times F^n \rightarrow F
\]
be a split $F/F_0$-hermitian form, i.e. there is a basis $e_1, \dots, e_n$ of $F^n$ such that
\[
h(ae_i,be_{n+1-j}) = \overline{a}b\cdot \delta_{i,j} \quad \text{for  all} \quad a,b \in F, 
\]
where $a \mapsto \overline{a}$ is the non-trivial element of $\text{Gal}(F/{F_0})$. Attached to $h$ are the respective alternating and symmetric $F_0$-bilinear forms $F^n \times F^n \rightarrow F_0$ given by
\[
\langle x, y \rangle = \frac{1}{2}\text{Tr}_{F /F_0} (\pi^{-1}\phi(x, y)) \quad \text{and} \quad ( x, y ) =  \frac{1}{2}\text{Tr}_{F /F_0} (\phi(x, y)).
\]
For any $O_F$-lattice $\Lambda$ in $F^n$, we denote by $\Lambda^\sharp = \{v \in F^n | h( v, \Lambda ) \subset O_{F} \}$, $\Lambda^\vee = \{v \in F^n | \langle v, \Lambda \rangle \subset O_{F_0} \}$ and $\Lambda^\bot = \{v \in F^n | ( v, \Lambda ) \subset O_{F_0} \}$ the dual lattices respectively for the hermitian, alternating and symmetric forms.
The forms $\langle \, , \, \rangle $ and $ (\, , \,) $
induce perfect $O_{F_0}$-bilinear pairings
\begin{equation}\label{perfectpairing}
    \Lambda \times \Lambda^\vee \xrightarrow{\langle \, , \, \rangle } O_{F_0}, \quad \Lambda^\bot \times \Lambda \xrightarrow{ (\, , \,)} O_{F_0}
\end{equation}
for all $\Lambda$. Then we have $\Lambda^\sharp= \Lambda^\vee=\pi \Lambda^\bot$.
For $i= kn+j$ with $0\leq j< n$, we define the standard lattices
\begin{equation}\label{eq 212}
\Lambda_i = \pi^{-k}\cdot \text{span}_{O_F} \{\pi^{-1}e_1, \dots, \pi^{-1}e_j, e_{j+1}, \dots, e_n\}.	
\end{equation} 
Note that $ \Lambda_{n-i}:=\Lambda^\bot_i $ and $ \Lambda_{-i}:=\Lambda^\sharp_i=\Lambda^\vee_i $. Then $\Lambda_i$'s form a self-dual periodic lattice chain $\mathcal{L}=\{\Lambda_i\}_{i\in \ZZ}$.
For nonempty subsets $I \subset \{0, \dots, m\}$ where $m = \lfloor n/2\rfloor$, 
let $\calL_I=\{\Lambda_i\}_{i\in \pm I+n\ZZ}$  be a self-dual periodic lattice chain. Let $\scrG_I = \underline{{\rm Aut}}(\mathcal{L}_I)$ be the (smooth) group scheme over $O_{F_0}$ with $P_I = \scrG_I(O_{F_0})$ the subgroup of $G(F_0)$ fixing the lattice chain $\mathcal{L}_I$ (see \cite[\S 1.2.3(b)]{PR} for more details). For even integers $2h, 2t$, where $0\leq 2h\neq 2t\leq n$, we define the following index sets: $[2h]=\pm h+n\ZZ$, $[2h, 2t]=\{\pm h, \pm t\}+n\ZZ$ and let $\calL_{[2h]}$, $\calL_{[2h,2t]}$ be the standard self-dual lattice chains.

We first recall the definition of {\it wedge local models} $\M^{[2h],\wedge}_n$. For any $O_F$-algebra $R$, let $\Lambda_{i,R}$ be the tensor product $\Lambda_i\otimes_{O_{F_0}} R$ as an $O_F\otimes_{O_{F_0}} R$-module. Set $\Pi=\pi\otimes 1$, and $\pi=1\otimes \pi$.

\begin{Definition}
{\rm
The wedge local model $\M^{[2h],\wedge}_n$ is a projective scheme over $\Spec O_F$ representing the functor that sends each $O_F$-algebra $R$ to the set of subsheaves $\calF_i \subset  \Lambda_{i,R}$, where $i\in [2h]$, such that
\begin{itemize}
    \item For all $i\in [2h]$, $\calF_i$ as $O_F\otimes R$-modules are Zariski locally on $R$ direct summands of rank $n$.
    \item For all $i, j\in [2h]$ with $i<j$, the maps induced by the inclusions $  \Lambda_{i,R}\subset \Lambda_{j, R}$ restrict to maps
    \[
    \calF_i \rightarrow \calF_j.
    \]
    \item For all $i\in [2h]$, the isomorphism $\Pi: \Lambda_{i, R} \rightarroweq \Lambda_{i-n, R}$ identifies
    \[
    \calF_i \rightarroweq \calF_{i-n}.
    \]
    \item For all $i\in [2h]$, $\calF_{-i}$ is the orthogonal complement of $\calF_i$ with respect to $ \bb \,,\,\pp : \Lambda_{-i} \times  \Lambda_i\rightarrow R$.
    \item (Kottwitz condition) For all $i\in [2h]$,
    \[
    \text{char}_{\Pi |  \mathcal{F}_i} (X)= (X + \pi)^{n-1}(X - \pi) .
    \] 
    \item (Wedge condition) For all $i\in [2h]$,
    \[
    \wedge^2(\Pi-\pi\mid \calF_i)=0, \quad \wedge^n(\Pi+\pi\mid \calF_i)=0.
    \]
\end{itemize}   
}
\end{Definition}

Wedge local models $\M^{[2h],\wedge}_n$ are not always flat in general (see \cite{Sm3} for more details). We define the {\it local model} $\M^{\loc, [2h]}_n$ as the flat closure of $\M^{[2h],\wedge}_n$. Then the strata local models $\M^{\loc, [2h]}_n(2t)$ are defined as follows:

\begin{Definition}
{\rm
Let $R$ be an $k$-algebra, and $\calL_{[2h]}$, $\calL_{[2h,2t]}$ be the standard self-dual lattice chains. The strata local model $\M^{\loc, [2h]}_n(2t)$ is the projective scheme over $\Spec k$, representing the functor that sends each $k$-algebra $R$ to the set of subsheaves 
\[
\calF_i \subset  \Lambda_{i,R}, \quad \text{where}~ i\in [2h,2t] 
\]
such that
\begin{itemize}
    \item $\calF_i=\Pi \Lambda_i$ for $i\in [2t]$.
    \item $(\calF_i)_{i\in [2h]}\in \overline{\M}^{\loc, [2h]}_n\otimes R$.
    \item For any $i< j$ with either $i\in [2h], j\in [2t]$ or $i\in [2t], j\in [2h]$, the natural morphism $\Lambda_i\rightarrow \Lambda_j$ maps $\calF_i$ to $\calF_j$.
\end{itemize}
}
\end{Definition}

Here $\overline{\M}^{\loc, [2h]}_n$ is the special fiber of the local model $\M^{\loc, [2h]}_n$. The main theorem for the strata local model is the following:

\begin{Theorem}[\cite{HLS2}, \S 4]\label{StrataLocalModel}
The strata local model  $\M^{\loc, [2h]}_n(2t)$ is reduced, normal, and Cohen-Macaulay. Moreover, 

(1). For $t> h$, the strata local model $\M^{\loc, [2h]}_n(2t)$ has dimension $t+h$.

(2). For $t< h$, excluding the case where $n$ is even and $h=\frac n2$ ($\pi$-modular case), the strata local model $\M^{\loc, [2h]}_n(2t)$ has dimension $n-t-h-1$.

(3). For $n$ is even, $t< h=\frac n2$, the strata local model $\M^{\loc, [n]}_n(2t)$ is smooth and irreducible of dimension $\frac n2-t-1$.
\end{Theorem}

\subsection{Strata splitting models}\label{StrataSpl}
In order to define the strata splitting models, we first recall the definition of splitting models.

\begin{Definition}
{\rm
The splitting model $\M^{\spl,[2h]}_n$ is a projective scheme over $\Spec O_F$ representing the functor that sends each $O_F$-algebra $R$ to the set of subsheaves 
\[
\begin{array}{l}
   \calF_i \subset  \Lambda_{i,R}, \quad ~\text{where}~ i\in [2h]  \\
   \calG_j \subset  \calF_j, \quad ~\text{where}~ j\in \{-h\}+n\ZZ
\end{array}
\]
such that
\begin{itemize}
    \item For all $i\in [2h], j\in \{-h\}+n\ZZ$, $\calF_i$ (resp. $\calG_j$) as $O_F\otimes R$-modules are Zariski locally on $R$ direct summands of rank $n$ (resp. rank $1$).
    \item For all $i\in [2h]$, $(\calF_i)\in \M^{[2h],\wedge}_n\otimes R$.
    \item (Splitting condition)  For all $j\in \{-h\}+n\ZZ$,
    \[
    (\Pi+\pi) \calF_{j} \subset \calG_{j}, \quad  (\Pi-\pi) \calG_{j} = (0).
    \]
\end{itemize}
}
\end{Definition}

We recall the following facts about the splitting model.
\begin{Proposition}[\cite{ZacZhao2}, Theorem 5.1]\label{SplitFlat}

a) The scheme $\M^{\spl,[2h]}_n$ is $O_F$-flat, normal and Cohen-Macaulay.

b) The special fiber of $\M^{\spl,[2h]}_n$ is reduced.  
\end{Proposition}
The splitting model $\M^{\spl,[2h]}_n$ supports an action of $\scrG_{[2h]}$ and there is a $\scrG_{[2h]}$-equivariant projective morphism 
\[
 \tau : \M^{\spl,[2h]}_n \rightarrow \M^{\loc,[2h]}_n
\]
which is given by $(\calF_i, \calG_{j}) \mapsto (\calF_i)$ on $R$-valued points. (Indeed, we can easily see, as in \cite[Definition 4.1]{Kr}, that $\tau$ is well defined.) The morphism $\tau : \M^{\spl,[2h]}_n \rightarrow \M^{\loc,[2h]}_n$ induces an isomorphism on the generic fibers (see \cite[\S 3.2]{ZacZhao2}). 

\begin{Remark}\label{Case_m-1}{\rm
The case $n=2m$ is even and $h=m-1$ is not directly treated in \cite{ZacZhao2}. However, we can follow exactly the same steps and obtain the affine charts described in \cite[Proposition 4.2]{ZacZhao2} for $n=2m$ and $ h = m-1$ ($h$ corresponds to $\ell$ in loc. cit.). Then, it is an easy exercise to verify that these affine charts, and so the splitting model $ \M^{\spl,[2(m-1)]}_{2m}$, are flat, normal, Cohen-Macaulay and with a reduced special fiber. 
}\end{Remark}

Now we can define the strata splitting models. For even integers $2h, 2t$, where $0\leq 2h\neq 2t\leq n$, let $\overline{\M}^{\spl,[2h]}_n$ be the special fiber of the splitting model over $\Spec k$.

\begin{Definition}
{\rm
Let $R$ be an $k$-algebra. The strata splitting model $\M^{\spl,[2h]}_n(2t)$ is the projective scheme over $\Spec k$, representing the functor that sends each $k$-algebra $R$ to the set of subsheaves 
\[
\begin{array}{l}
   \calF_i \subset  \Lambda_{i,R}, \quad ~\text{where}~ i\in [2h, 2t]  \\
   \calG_j \subset  \calF_i, \quad ~\text{where}~ j\in \{-h\}+n\ZZ
\end{array}
\]
such that
\begin{itemize}
    \item $\calF_i=\Pi \Lambda_i$ for $i\in [2t]$.
    \item $(\calF_i, \calG_j)_{i\in [2h], j\in \{-h\}+n\ZZ}\in \overline{\M}^{\spl, [2h]}_n\otimes R$.
    \item For any $i_1< i_2$ with either $i_1\in [2h], i_2\in [2t]$ or $i_1\in [2t], i_2\in [2h]$, the natural morphism $\Lambda_{i_1}\rightarrow \Lambda_{i_2}$ maps $\calF_{i_1}$ to $\calF_{i_2}$.
    \item Let $j=-h+kn$ for some $k$, $\calG_j$ satisfies the following condition:
    
    (1).When $t> h$, we have $\calG_j\subset \Lambda_M^\bot$, where $\Lambda_M$ is the image of $\Lambda_{-t+kn}\rightarrow \Lambda_{h+kn}$, and $\Lambda_M^\bot$ is the dual of $\Lambda_M$ with respect to $\bb~,~ \pp$.

    (2).When $t< h$, we have $\calG_j\subset \Lambda_M^\bot$, where $\Lambda_M$ is the image of $\Lambda_{t+(k-1)n}\rightarrow \Lambda_{h+(k-1)n}$, and $\Lambda_M^\bot$ is the dual of $\Lambda_M$ with respect to $( ~ ,~ )$.
\end{itemize}
}
\end{Definition}

The strata splitting model $\M^{\spl,[2h]}_n(2t)$ is a  closed subscheme of $\overline{\M}^{\spl,[2h]}_n$. By restricting to the strata splitting models, we get the projective $\scrG_{[2h]}$-equivariant morphism $ \tau: \M^{\spl,[2h]}_n(2t)\rightarrow \M^{\loc,[2h]}_n(2t)$. The main theorem of this section is as follows.

\begin{Theorem}\label{thm splstrata}
The strata splitting model  $\M^{\spl, [2h]}_n(2t)$ is reduced, normal, and Cohen-Macaulay. Moreover, 

(1). For $t> h$, the strata splitting model $\M^{\spl, [2h]}_n(2t)$ has dimension $t+h$.

(2). For $t< h$, excluding the case where $n$ is even and $h=\frac n2$ ($\pi$-modular case), the strata splitting model $\M^{\spl, [2h]}_n(2t)$ is smooth of dimension $n-t-h-1$.

(3). For $t_2<h<t_1$, the intersection of strata splitting models $\M^{\spl, [2h]}_n(2t_1)\cap \M^{\spl, [2h]}_n(2t_2)$ is smooth of dimension $t_1-t_2-1$.
\end{Theorem}

\subsection{Affine charts}

In this subsection, we will prove Theorem \ref{thm splstrata}. Using the construction of the strata local models (\cite[\S 4]{HLS2}) and the projective $\scrG_{[2h]}$-equivariant morphism $ \tau: \M^{\spl,[2h]}_n(2t)\rightarrow \M^{\loc,[2h]}_n(2t)$, it suffices to compute an open affine covering of the inverse image of the worst point under $\tau$. We refer the reader to \cite[\S 4.1]{ZacZhao2} for the definition of the worst point. 
Here, studying these affine charts will yield Propositions \ref{prop X-strata} and \ref{prop Y-strata}.

\subsubsection{$\calZ^\spl$-strata splitting models}
Let $\Lambda\subset C$ be a vertex lattice of type $2t$ with $t> h$. We begin by computing the strata splitting model associated to $\calZ(\Lambda)$. By an unramified extension, we can reduce to the case where the hermitian form is split. We select the same affine charts as those in \cite[\S 4.2]{ZacZhao2}, i.e.,
\begin{equation}
\begin{tikzcd}
\Lambda_{-t,R} \arrow[r,"\lambda_1"] 
&
\Lambda_{-h,R} \arrow[r,"\lambda"] 
&
\Lambda_{h,R} \arrow[r,"\lambda_2"]
&
\Lambda_{t,R}
\\
\pi\Lambda_{-t,R} \arrow[u,hook]\arrow{r}
&
\calF_{-h} \arrow[u,hook]\arrow{r}
&
\calF_{h} \arrow[u,hook]\arrow{r}
&
\pi\Lambda_{t,R}\arrow[u,hook]
\\
&
\calG_{-h}\arrow[u, hook]
&&
\end{tikzcd},
\end{equation}
where
\begin{equation}\label{eq FG}
\calF_h = \left[\ 
\begin{matrix}
X\\ \hline
I_n  
\end{matrix}\ \right], \quad 
\mathcal{F}_{-h}=  \left[\ 
\begin{matrix}
Y\\ \hline
I_n  
\end{matrix}\ \right],\quad
\calG_{-h}= \left[\ 
\begin{matrix}
0\\ \hline
V  
\end{matrix}\ \right]   
\end{equation}
and $X, Y$ are matrices of size $n\times n$ and the matrix $V$ is of size $n\times 1$. We break up the matrices $X, Y, V$ into blocks as follows:
\begin{equation}\label{eq XYV}
X=\left[\ \begin{matrix}
X_1&X_2 \\ 
X_3&X_4 
\end{matrix}\ \right],\quad
Y=\left[\ \begin{matrix}
Y_1&Y_2 \\ 
Y_3&Y_4 
\end{matrix}\ \right],	\quad
V= \left[\ 
\begin{matrix}
V_1\\ 
V_2 
\end{matrix}\ \right],
\end{equation}
where $X_1$ and $Y_1$ are of size $2h\times 2h$, $X_4$ and $Y_4$ are of size $(n-2h)\times (n-2h)$, and $V_1$ and $V_2$ are of sizes $2h\times 1$ and $(n-2h)\times 1$ respectively. By \cite[Lemma 4.2.1]{ZacZhao2}, there exists a $n\times 1$ matrix $Z= \left[\ 
\begin{matrix}
Z_1\\ 
Z_2 
\end{matrix}\ \right]$, with $Z_i$ of the same size as $Y_i$ ($i=1,2$), such that $X, Y$ can be expressed in terms of $V, Z$:
\begin{equation}\label{eq 334}
\begin{array}{llll}
Y_1=V_1Z_1^t, & Y_2=V_1Z_2^t,& Y_3=V_2Z_1^t, & Y_4=V_2Z_2^t,\\
X_1=-JZ_1V_1^tJ, &X_2=JZ_1V_2^tH, &X_3=-HZ_2V_1^tJ, & X_4=HZ_2V_2^tH.
\end{array}
\end{equation}	
Here $H=H_{n-2h}$ is the unit anti-diagonal matrix of size $n-2h$, and $J=\left[\ 
\begin{matrix}
&H_h \\ 
-H_h& 
\end{matrix}\ \right]$ of size $2h$. Moreover, by \cite[Proposition 4.2.2]{ZacZhao2}, the matrices $V_1, V_2, Z_1, Z_2$ satisfy the following conditions:
\begin{equation}\label{eq 344}
Z_1=-\frac12 (Z_2^tHZ_2)JV_1, \quad
\wedge^2 (V_2\mid HZ_2)=0,\quad
Z_2^tV_2=0.
\end{equation}
To make $(\calF_h, \calF_{-h}, \calG_{-h})\in \M^{\spl, [2h]}_n(2t)(R)$, we still need to check:
\begin{equation}
\lambda_1(\Pi \Lambda_{-t,R})\subset \calF_{-h},\quad
\lambda_2(\calF_h)\subset \Pi \Lambda_{t,R},\quad
\calG_{-h}\subset \Lambda_{M}^\bot, 
\end{equation}
where $\Lambda_M$ is the image of $\Lambda_{-t}\rightarrow \Lambda_h$. Note that the ordered basis of $\Lambda_{\pm t}$, $\Lambda_{\pm h}$ are the same as \cite[(4.1.1)]{ZacZhao2}. With respect to these ordered basis, we have:
\begin{equation}
\lambda_1=\left[\ \begin{array}{ll|ll}
0 & I_{n-t+h} &0 & 0 \\ 
0_{t-h}& 0 &0 &0\\ \hline
0&0_{n-t+h}&0&I_{n-t+h}\\
I_{t-h}&0&0_{t-h}&0
\end{array}\ \right],\quad	
\lambda_2=\left[\ \begin{array}{ll|ll}
0 & I_{t-h} &0 & 0 \\ 
A& 0 &0 &0\\ \hline
0&0_{t-h}&0&I_{t-h}\\
B&0&A&0
\end{array}\ \right],
\end{equation}
where
\[
A=\left[\ \begin{matrix}
I_{2h}  &0 & 0 \\ 
 0 &0_{t-h} &0\\
0&0&I_{n-2t}
\end{matrix}\ \right],\quad
B=\left[\ \begin{matrix}
0_{2h}  &0 & 0 \\ 
 0 &I_{t-h} &0\\
0&0&0_{n-2t}
\end{matrix}\ \right].
\]
Now $\lambda_1(\Pi \Lambda_{-t,R})\subset \calF_{-h}$ implies that
\begin{equation}\label{eq Y}
Y_1=0,\quad Y_3=0,\quad Y_2=\begin{blockarray}{cccc}
\matindex{t-h} &\matindex{n-2t}&\matindex{t-h}&\\
\begin{block}{[ccc]c}
  0  &0 & * & \matindex{h} \\ 
 0 &0&*& \matindex{h} \\
\end{block}
   \end{blockarray}
,\quad
Y_4=\begin{blockarray}{cccc}
\matindex{t-h} &\matindex{n-2t}&\matindex{t-h}&\\
\begin{block}{[ccc]c}
  0  &0 & * & \matindex{t-h} \\ 
 0 &0&*& \matindex{n-2t} \\
 0  &0 & * & \matindex{t-h} \\ 
\end{block}
   \end{blockarray}.
\end{equation}
The coordinates of $V_1, V_2, Z_1, Z_2$ can be further refined as follows:
\begin{equation}
V_1= \begin{blockarray}{cc}
\matindex{1} & \\
\begin{block}{[c]c}
V_{1,1}  & \matindex{h} \\ 
V_{1,2}  & \matindex{h} \\ 
\end{block}
   \end{blockarray},
V_2= \begin{blockarray}{cc}
\matindex{1} & \\
\begin{block}{[c]c}
V_{2,1}  & \matindex{t-h} \\ 
V_{2,2}  & \matindex{n-2t} \\ 
V_{2,3}  & \matindex{t-h} \\ 
\end{block}
   \end{blockarray},
Z_1= \begin{blockarray}{cc}
\matindex{1} & \\
\begin{block}{[c]c}
Z_{1,1}  & \matindex{h} \\ 
Z_{1,2}  & \matindex{h} \\ 
\end{block}
   \end{blockarray},
Z_2= \begin{blockarray}{cc}
\matindex{1} & \\
\begin{block}{[c]c}
Z_{2,1}  & \matindex{t-h} \\ 
Z_{2,2}  & \matindex{n-2t} \\ 
Z_{2,3}  & \matindex{t-h} \\ 
\end{block}
   \end{blockarray}.
\end{equation}
Thus, we obtain 
\begin{equation}\label{eq 349}
\left[\ 
\begin{matrix}
V_1\\ 
V_2
\end{matrix}\ \right]\cdot 
\left[\ 
\begin{matrix}
Z_{2,1}^t, Z_{2,2}^t\\ 
\end{matrix}\ \right]=0    
\end{equation}
by combining equations of $Y_2, Y_4$ in (\ref{eq 334}), (\ref{eq Y}). Since there exists a unit element in $V$, equation (\ref{eq 349}) is equivalent to 
\begin{equation}\label{eq 3410}
Z_{2,1}=0,\quad Z_{2,2}=0.
\end{equation}
Note that using the above relations we deduce $Z_2^tHZ_2=0$. We have $Z_1=-\frac12 (Z_2^tHZ_2)JV_1=0$, and $Y_1=V_1Z_1^t=0, Y_3=V_2Z_1^t=0$.

Similarly, condition $\lambda_2(\calF_h)\subset \Pi \Lambda_{t,R}$ is equivalent to 
\begin{equation}\label{eq X}
X_1=0,\quad X_2=0,\quad X_3=\left[\ \begin{matrix}
* & * \\ 
 0 &0\\
 0& 0
\end{matrix}\ \right],\quad
X_4=\left[\ \begin{matrix}
*  &*& * \\ 
 0 &0 &0\\
0&0&0
\end{matrix}\ \right]. 
\end{equation}
It is not hard to check that we get the same equations as in (\ref{eq 3410}).

Finally, consider $\calG_{-h}\subset \Lambda_{M}^\bot$, where $\Lambda_M$ is the image of $\Lambda_{-t,R}\rightarrow \Lambda_{h,R}$, i.e.,
\begin{equation}
\Lambda_M=\text{span}_{R}\{e_1,\cdots,e_{n-t}, \pi e_{h+1},\cdots, \pi e_n\}\subset \Lambda_h,
\end{equation}
of rank $2n-(h+t)$. The dual of $\Lambda_M$ with respect to $\bb ~ , ~\pp$ is
\begin{equation}
\Lambda_M^\bot=\text{span}_{R}\{\pi e_1,\cdots,\pi e_{t}, p e_{n-h+1},\cdots, p e_n\}\subset \Lambda_{-h},
\end{equation}
of rank $t+h$. Reordering the basis as in \cite[(4.1.1)]{ZacZhao2}, we have:
\begin{equation}
\calG_{-h}= \left[\ 
\begin{matrix}
0_n\\ \hline
V_{1,1}\\
V_{1,2}\\
V_{2,1}\\
V_{2,2}\\
V_{2,3}\\
\end{matrix}\ \right]\subset
\Lambda_M^\bot=\left[\begin{matrix}
&0_n&\\ \hline
I_h &&\\
&I_h&\\
&& I_{t-h}\\
0&0&0\\
0&0&0\\
\end{matrix}\ \right].
\end{equation}
Therefore, we get $V_{2,2}=0, V_{2,3}=0$.

From the above, by fixing some element $v_{i_0}=1$ in $\calG_{-h}$, we have:

\begin{Proposition}\label{prop X-strata}
The open affine chart $\calU_{i_0}$ in $\calZ^\spl$-strata splitting model $\M^{\spl, [2h]}_n (2t)$ $(t> h)$ is isomorphic to 
\[
\calU_{i_0}\simeq  \Spec \frac{k[V_{1,1}, V_{1,2}, V_{2,1}, Z_{2,3}]}{(v_{i_0}-1, \wedge^2(V_{2,1}\mid HZ_{2,3}))}.
\]
The affine chart $U_{i_0}$ is reduced, normal, and Cohen-Macaulay. We have

(1) When $t>h=0$, $U_{i_0}\simeq \mathbb{A}_k^{t}$ is smooth of dimension $t$;

(2) When $h>0$ and $t-h=1$, $U_{i_0}\simeq \mathbb{A}_k^{2h+1}$ is smooth of dimension $t+h$;

(3) When $h>0$ and $t-h\geq 2$, $U_{i_0}$ is singular of dimension $t+h$.
\end{Proposition}

\begin{proof}
By (\ref{eq 344}) and (\ref{eq 3410}), we have $Z_1=0$, and equation $\wedge^2 (V_2\mid HZ_2)=0$ is equivalent to $\wedge^2(V_{2,1}\mid HZ_{2,3})$ since $V_{2,2}=0, V_{2,3}=0$. Note that $Z_2^tV_2=\sum_{i=1}^3 Z_{2,i}^t V_{2,i}=0$ is automatically satisfied. Thus, the only non-zero matrices are $V_{1,1}, V_{1,2}, V_{2,1}, Z_{2,3}$ with  $\wedge^2(V_{2,1}\mid HZ_{2,3})$. There is a unit element in the matrix $[V_{1,1}^t~V_{1,2}^t~ V_{2,1}^t]^t$, denoted by $v_{i_0}$. Therefore, the open affine chart $\calU_{i_0}$ associated to $v_{i_0}$ is:
\[
\Spec \frac{k[V_{1,1}, V_{1,2}, V_{2,1}, Z_{2,3}]}{(v_{i_0}-1, \wedge^2(V_{2,1}\mid HZ_{2,3}))}. 
\]
This variety is flat, normal, and Cohen-Macaulay by \cite[(2.1.1)]{BV}. When $t>h=0$, the matrices $V_{1,1}$, $V_{1,2}$ are empty and so the matrix $V_{2,1}$ should contain the unit element $v_{i_0}$. Then from $\wedge^2(V_{2,1}\mid HZ_{2,3})$ we get that the matrix $Z_{2, 3}$ is determined by $V_{2,1}$ and one parameter. Thus, we deduce that $U_{i_0}\simeq \mathbb{A}_k^{t}$ is smooth. When $h>0$ and $t-h=1$, the condition $\wedge^2(V_{2,1}\mid HZ_{2,3})$ becomes trivial and $U_{i_0}$ is also smooth. This finishes the proof of the proposition.
\end{proof}

\begin{Remark}
{\rm
(1). In the above proposition, we show that when $t>h=0$, $U_{i_0}$ is smooth of dimension $t$ and so is the strata splitting model $\M^{\spl, [0]}_n(2t)$. In the next section, we will prove that the Bruhat-Tits strata $\calZ^{\spl}(\Lambda)$ are \'etale locally isomorphic to $\M^{\spl, [0]}_n(2t)$. Thus, our results in Proposition \ref{prop X-strata} (1) recover the BT-strata described in \cite{HLSY}.

(2). Consider the case $n=2m$ is even and $h=m-1$. For the strata splitting model $\M^{\spl, [2h]}_n (2t)$, we have $t=m$ by $h< t\leq \frac n2$. Then the affine chart $\calU_{i_0}$ of $\M^{\spl, [n-2]}_{n} (n)$ is isomorphic to $\mathbb{A}^{n-1}_{k}$ by Proposition \ref{prop X-strata}.
}
\end{Remark}

\subsubsection{$\calY^\spl$-strata splitting models}
Now we consider the $\calY^\spl$-strata splitting models. Let $\Lambda\subset C$ be a vertex lattice of type $2t$ with $t< h$. (Recall from Remark \ref{Excludemodular} that we exclude the $\pi$-modular case, i.e. even $n$ and $h=\frac n2$.) 
The affine charts parameterize lattice chains:
\begin{equation}
\begin{tikzcd}
\Lambda_{t,R} \arrow[r,"\lambda_2"] 
&
\Lambda_{h,R} \arrow[r,"\lambda^\vee"] 
&
\Lambda_{n-h,R} \arrow[r,"\lambda_1"]
&
\Lambda_{n-t,R}
\\
\pi\Lambda_{t,R} \arrow[u,hook]\arrow{r}
&
\calF_{h} \arrow[u,hook]\arrow{r}
&
\calF_{n-h} \arrow[u,hook]\arrow{r}
&
\pi\Lambda_{n-t,R}\arrow[u,hook]
\\
&&\calG_{n-h}\arrow[u, hook]&
\end{tikzcd}.
\end{equation}
We break up $\calF_h$, $\calF_{n-h}\simeq \calF_{-h}$ and $\calG_{n-h}\simeq \calG_{-h}$ into the same matrix blocks as in (\ref{eq FG}), (\ref{eq XYV}). Since the matrices $V_1, V_2, Z_1, Z_2$ satisfy the same conditions:
\begin{equation}\label{eq 3317}
Z_1=-\frac12 (Z_2^tHZ_2)JV_1, \quad
\wedge^2 (V_2\mid HZ_2)=0,\quad
Z_2^tV_2=0,
\end{equation}
we only need to check:
\begin{equation}
\lambda_2(\Pi \Lambda_{t,R})\subset \calF_{h},\quad
\lambda_1(\calF_{n-h})\subset \Pi \Lambda_{n-t,R},\quad
\calG_{n-h}\subset \Lambda_{M}^\bot.
\end{equation}
Here $\Lambda_M$ is the image of $\Lambda_{t}\rightarrow \Lambda_h$, and $\Lambda_M^\bot$ is the dual lattice with respect to $( ~ , ~)$. The transition maps can be expressed as:
\begin{equation}
\lambda_1=\left[\ \begin{array}{ll|ll}
0 & I_{n-h+t} &0 & 0 \\ 
0_{h-t}& 0 &0 &0\\ \hline
0&0_{n-h+t}&0&I_{n-h+t}\\
I_{h-t}&0&0_{h-t}&0
\end{array}\ \right],\quad	
\lambda_2=\left[\ \begin{array}{ll|ll}
0 & I_{h-t} &0 & 0 \\ 
A& 0 &0 &0\\ \hline
0&0_{h-t}&0&I_{h-t}\\
B&0&A&0
\end{array}\ \right]
\end{equation}
where
\[
A=\left[\ \begin{matrix}
I_{2t}  &0 & 0 \\ 
 0 &0_{h-t} &0\\
0&0&I_{n-2h}
\end{matrix}\ \right],\quad
B=\left[\ \begin{matrix}
0_{2t}  &0 & 0 \\ 
 0 &I_{h-t} &0\\
0&0&0_{n-2h}
\end{matrix}\ \right].
\]
Condition $\lambda_1(\calF_{n-h})\subset \Pi \Lambda_{n-t,R}$ is equivalent to
\begin{equation}\label{eq Y2}
Y_1=\begin{blockarray}{ccccc}
\matindex{h-t} &\matindex{t}&\matindex{t}& \matindex{h-t}&\\
\begin{block}{[cccc]c}
  *  &* & * & *& \matindex{h-t} \\ 
 0 &0&0&0& \matindex{t} \\
 0  &0 & 0 & 0& \matindex{t} \\ 
 0 &0&0&0& \matindex{h-t} \\
\end{block}
   \end{blockarray},\quad
Y_2=\begin{blockarray}{cc}
\matindex{n-2h} &\\
\begin{block}{[c]c}
  * & \matindex{h-t} \\ 
 0 & \matindex{t} \\
 0 & \matindex{t} \\ 
 0 & \matindex{h-t} \\ 
\end{block}
   \end{blockarray},\quad
Y_3=0,\quad Y_4=0.
\end{equation}
The coordinates of $V_1, Z_1$ can be further refined as follows:
\begin{equation}
V_1= \begin{blockarray}{cc}
\matindex{1} & \\
\begin{block}{[c]c}
V'_{1,1}  & \matindex{h-t} \\ 
V'_{1,2}  & \matindex{t} \\ 
V'_{1,3}  & \matindex{t} \\ 
V'_{1,4}  & \matindex{h-t} \\ 
\end{block}
   \end{blockarray},\quad
Z_1= \begin{blockarray}{cc}
\matindex{1} & \\
\begin{block}{[c]c}
Z'_{1,1}  & \matindex{h-t} \\ 
Z'_{1,2}  & \matindex{t} \\ 
Z'_{1,3}  & \matindex{t} \\ 
Z'_{1,4}  & \matindex{h-t} \\ 
\end{block}
   \end{blockarray}.
\end{equation}
Note that $Y_1=V_1Z_1^t$, and $Z_1=-\frac12 (Z_2^tHZ_2)JV_1$ by (\ref{eq 3317}). We obtain $Y_1=\frac 12 (Z_2^tHZ_2) V_1V_1^tJ$, such that $Y_1=JY_1^tJ$. This implies $Y_1=(y_{i,j})_{1\leq i,j\leq 2h}$ satisfying $y_{i,j}=\pm y_{2h+1-j, 2h+1-i}$ for $1\leq i, j\leq 2h$. Thus, the matrix $Y_1$ can be rewritten as 
\[
Y_1=\begin{blockarray}{ccccc}
\matindex{h-t} &\matindex{t}&\matindex{t}& \matindex{h-t}&\\
\begin{block}{[cccc]c}
  0  &0 & 0 & *& \matindex{h-t} \\ 
 0 &0&0&0& \matindex{t} \\
 0  &0 & 0 & 0& \matindex{t} \\ 
 0 &0&0&0& \matindex{h-t} \\
\end{block}
   \end{blockarray}.
\]
From the first three columns of $Y_1, Y_3$, we get
\begin{equation}\label{eq 3322}
\left[\ 
\begin{matrix}
V_1\\ 
V_2
\end{matrix}\ \right]\cdot 
\left[\ 
\begin{matrix}
(Z_{1,1}')^{t} ~ (Z_{1,2}')^{t} ~(Z_{1,3}')^{t}\\ 
\end{matrix}\ \right]=0.  
\end{equation}
Since there exists a unit element in $V$, equation (\ref{eq 3322}) is equivalent to $Z_{1,1}'=0$, $Z_{1,2}'=Z_{1,3}'=0$. From $Z_1=-\frac12 (Z_2^tHZ_2)JV_1$, we then have 
\begin{equation}\label{eq 3323}
(Z_2^tHZ_2)V_{1,i}'=0 ~\text{for}~ i=2,3,4.
\end{equation}
Equations of $Y_2, Y_4$ are equivalent to $Y_2=V_2Z_2^t=0$, $V_{1,i}'Z_2^t=0$ for $i=2,3,4$. Since $Z_2^tHZ_2$ is an element of $1\times 1$, we can represent $(Z_2^tHZ_2)V_{1,i}'$ as $(Z_2^tHZ_2)V_{1,i}'=V_{1,i}'(Z_2^tHZ_2)$, so that (\ref{eq 3323}) are automatically satisfied by $V_{1,i}'Z_2^t=0$.
So far, we have relations
\begin{equation}\label{eq 3324}
V_2Z_2^t=0,\quad V_{1,i}'Z_2^t=0, 
\end{equation}
for $i=2,3,4$. It is easy to see that $V_2Z_2^t=0$ implies that $\wedge^2 (V_2\mid HZ_2)=0$, and $Z_2^tV_2=\mathrm{Tr}(V_2Z_2^t)=0$ in (\ref{eq 3317}).

Similarly, condition $\lambda_2(\Pi \Lambda_{t,R})\subset \calF_{h}$ gives us the same relations as in (\ref{eq 3324}), so we only need to check $\calG_{n-h}\subset \Lambda_{M}^\bot$. Here $\Lambda_M$ is the image of $\Lambda_{t}\rightarrow \Lambda_h$. More precisely,
\begin{equation}
\Lambda_M=\text{span}_{R}\{\pi^{-1} e_1,\cdots,\pi^{-1} e_{t}, e_{1},\cdots, e_n, \pi e_{h+1}, \cdots, \pi e_n\}\subset \Lambda_{h}
\end{equation}
of rank $2n-(h-t)$. The dual of $\Lambda_M$ with respect to $( ~ , ~)$ is
\begin{equation}
\Lambda_M^\bot=\text{span}_{R}\{\pi e_{n-h+1}, \cdots, \pi e_{n-t}\}\subset \Lambda_{n-h},
\end{equation}
of rank $h-t$. Reordering the basis of $\Lambda_M^\bot$, condition $\calG_{n-h}\subset \Lambda_{M}^\bot$ is equivalent to 
\begin{equation}
V_{1,i}'=0,\quad V_2=0
\end{equation}
for $i=2,3,4$. Thus, the relations in (\ref{eq 3324}) are automatically satisfied. The only non-zero matrices are $V_{1,1}'$ and $Z_2$ with no relations between them. Therefore, we have the following proposition:

\begin{Proposition}\label{prop Y-strata}
The open affine chart $\calU_{i_0}$ in $\calY^\spl$-strata splitting model $\M^{\spl, [2h]}_n (2t)$ $(t<h)$ is isomorphic to 
\[
\calU_{i_0}\simeq  \Spec \frac{k[V_{1,1}', Z_2]}{(v_{i_0}-1)}\simeq \mathbb{A}_k^{n-h-t-1}.
\]
The affine chart $U_{i_0}$ is smooth for any $1\leq t< h\leq \lfloor \frac n2\rfloor$. 
\end{Proposition}

\subsubsection{Intersection of $\calZ^\spl$-strata and $\calY^\spl$-strata splitting models}\label{intersection}
Let $\Lambda_1\subset F^n$ be a vertex lattice of type $2t_1$ with $t_1> h$, and $\Lambda_2^\sharp\subset F^n$ be a vertex lattice of type $2t_2$ with $t_2< h$. Consider the intersection $\M^{\spl, [2h]}_n (2t_1)\cap \M^{\spl, [2h]}_n (2t_2)$. 

By Proposition \ref{prop X-strata} and \ref{prop Y-strata}, it is easy to see that $V_2=0, Z_1=0$, and
\begin{equation}
V_1= \begin{blockarray}{cc}
\matindex{1} & \\
\begin{block}{[c]c}
V_{1,1}'  & \matindex{h-{$t_2$}} \\ 
0  & \matindex{n-h+$t_2$} \\ 
\end{block}
   \end{blockarray},\quad
Z_2= \begin{blockarray}{cc}
\matindex{1} & \\
\begin{block}{[c]c}
0  & \matindex{n-$t_1$+h} \\ 
Z_{2,3}  & \matindex{$t_1$-h} \\ 
\end{block}
   \end{blockarray}.
\end{equation}
Thus, the affine chart $\calU_{i_0}$ is smooth and isomorphic to 
\begin{equation}
\mathbb{A}_k^{t_1-t_2-1}.    
\end{equation}
Combining the above with Propositions \ref{prop X-strata} and \ref{prop Y-strata}, we finish the proof of Theorem \ref{thm splstrata}.

\section{Local properties of Bruhat-Tits strata}\label{LPBT}
In this section, the goal is to obtain certain nice local properties (e.g. reducedness, normality) for the BT-strata $\mathcal{Z}^{\rm spl} (\Lambda),\, \mathcal{Y}^{\rm spl}(\Lambda^\sharp) $. To do this, we will relate these BT-strata with the strata splitting models via the local model diagram.

First, let us briefly recall the construction of such a local diagram for the BT-strata $\mathcal{Z}^{\rm loc} (\Lambda),\, \mathcal{Y}^{\rm loc}(\Lambda^\sharp) $ given in \cite[\S 4]{HLS2}. Assume that $\Lambda\subset C$ is a vertex lattice of type $2t$. 

\begin{enumerate}
\item For $t \geq h$, define $\tilde{\mathcal{Z}}^{\rm loc}(\Lambda)$ to be a projective formal scheme over $\bar{k}$ that represents the functor sending each $\bar{k}$-algebra $R$ to the set of tuples $(X, \iota, \lambda, \rho, f)$ where:
\begin{itemize}
    \item $(X, \iota, \lambda, \rho) \in \mathcal{Z}^{\rm loc} (\Lambda)(R)$, 
    \item $f$ is an isomorphism between the standard lattice chain $\calL_{[2h,2t],R} := \calL_{[2h,2t]}\otimes R$ and the lattice chain of de Rham realizations:
    \[
f: \begin{tikzcd}
\Lambda_{-t,R}\arrow{r} \arrow[d,"\sim"]
&
\Lambda_{-h,R} \arrow{r} \arrow[d,"\sim"]
&
\Lambda_{h,R} \arrow{r} \arrow[d,"\sim"]
&
\Lambda_{t,R} \arrow[d,"\sim"]
\\
D(X_{\Lambda}) \arrow{r}
&
D(X) \arrow{r} \arrow{r} 
&
D(X^{\vee})\arrow{r}
&
D(X_{\Lambda^{\sharp}})
\end{tikzcd}.
\]
\end{itemize}

\item For $t \leq h$, define $\tilde{\mathcal{Y}}^{\rm loc}(\Lambda^{\sharp})$ to be a projective formal scheme over $\bar{k}$ that represents the functor sending each $\bar{k}$-algebra $R$ to the set of tuples $(X, \iota, \lambda, \rho, f)$ where:
\begin{itemize}
    \item $(X, \iota, \lambda, \rho) \in \mathcal{Y}^{\rm loc} (\Lambda^{\sharp})(R)$, 
    \item $f$ is an isomorphism between the standard lattice chain $\calL_{[2h,2t],R} $ and the lattice chain of de Rham realizations:
    \[
f: \begin{tikzcd}
\Lambda_{t,R}\arrow{r} \arrow[d,"\sim"]
&
\Lambda_{h,R} \arrow{r} \arrow[d,"\sim"]
&
\Lambda_{n-h,R} \arrow{r} \arrow[d,"\sim"]
&
\Lambda_{n-t,R} \arrow[d,"\sim"]
\\
D(X_{\Lambda^\sharp}) \arrow{r}
&
D(X^{\vee}) \arrow{r} \arrow{r} 
&
D(X)\arrow{r}
&
D(X_{\pi^{-1}\Lambda})
\end{tikzcd}.
\]
\end{itemize}
\end{enumerate}
Recall that $\scrG_{[2h, 2t]}$ is the smooth group scheme of automorphisms of the lattice chain $ \calL_{[2h,2t]}$. We have the local model diagram 
 \begin{equation}\label{LMdiagram1}
\begin{tikzcd}
&\tilde{\mathcal{Z}}^{\rm loc}(\Lambda)\arrow[dl, "\psi_1"']\arrow[dr, "\psi_2"]  & \\
\mathcal{Z}^{\rm loc}(\Lambda)  &&  \M^{\loc, [2h]}_n(2t)
\end{tikzcd}
\end{equation}
where $\psi_1$ is a smooth $\scrG_{[2h, 2t], \bar{k}}$-torsor of relative dimension $\operatorname{dim}\scrG_{[2h, 2t], \bar{k}}$ and $\psi_2$ is a smooth morhism of relative dimension $\operatorname{dim}\scrG_{[2h, 2t], \bar{k}}$. (We get a similar local model diagram for $\mathcal{Y}^{\rm loc}(\Lambda^{\sharp})$ and $\calZ^\loc(\Lambda_1)\cap \mathcal{Y}^{\rm loc}(\Lambda^{\sharp}_2)$.) Here, $\psi_1$ is defined by forgetting the trivialization $f$ and $ \psi_2$ is defined by attaching the Hodge filtration of the strict $O_{F_0}$-modules to the lattice chain through the isomorphism $f$ (see \cite[\S 4.2]{HLS2} for more details). 

Recall from Section \ref{StrataSpl} that there is a projective morphism $\tau: \M^{\spl,[2h]}_n(2t)\rightarrow \M^{\loc,[2h]}_n(2t)$. From all the above, we deduce that $\mathcal{Z}^{\rm spl}(\Lambda)$ is a linear modification of $\mathcal{Z}^{\rm loc}(\Lambda)$ in the sense of \cite[\S 2]{P} and in particular there is a local model diagram for $\mathcal{Z}^{\rm spl}(\Lambda)$ similar to (\ref{LMdiagram1}) but with $  \M^{\loc,[2h]}_n(2t)$ replaced by $ \M^{\spl,[2h]}_n(2t)$. A similar local model diagram can be constructed for $\mathcal{Y}^{\rm spl}(\Lambda^{\sharp}) $.  Also, using analogous arguments a local model diagram can be constructed between $\calZ^\spl(\Lambda_1)\cap \mathcal{Y}^{\rm spl}(\Lambda^{\sharp}_2)$ and $\M^{\spl, [2h]}_n (2t_1)\cap \M^{\spl, [2h]}_n (2t_2)$ where $\Lambda_1$ is a vertex lattice of type $2t_1$ with $t_1> h$ and $\Lambda_2^\sharp$ is a vertex lattice of type $2t_2$ with $t_2< h$. 

\begin{Remark}
{\rm
It is worth mentioning that for $S=\Spec R$, with $R$ a $\bar{k}$-algebra, the condition $ x_{*} (\operatorname{Lie}(Y\times S)) \subset \mathcal{F}$ of $\calZ^{\spl}(\Lambda)(S)$ (see \S \ref{BT_strata}) is equivalent, via the local model diagram, to $\calG_{-h}\subset \Lambda_M^\bot$ in the strata splitting model $M^{\spl, [2h]}_n(2t)(R)$. Note that 
$ x_{*} (\operatorname{Lie}(Y\times S)) \subset \mathcal{F}$ translates to 
\[
\Lambda \otimes W_{O_{F_0}}(\kappa) \subset M'(X) \subset M(X)^\sharp
\]
in Proposition \ref{DieudLatt}. Recall that there is a perfect pairing $\operatorname{Fil} (X) \times \operatorname{Lie}(X^\vee)\to \mathcal{O}_S$ induced by (\ref{perf.pair.}). Let $\calF^\bot\subset \operatorname{Fil} (X)$ be the perpendicular complement of $\calF\subset \operatorname{Lie}(X^\vee)$. Identifying $D(X)=M(X)/\pi_0 M(X)$ with the standard lattice $\Lambda_{-h, R}$ and setting $\calG_{-h}\subset \Lambda_{-h, R}$ to be the lattice corresponding to $\calF^\bot$ we obtain $\Lambda_M\subset \calG_{-h}^\bot$, hence $\calG_{-h}\subset \Lambda_M^\bot$. Similarly, the condition $ x^{\sharp}_{*} (\operatorname{Lie}(Y\times S)) \subset \mathcal{F}$ translates to $\calG_{n-h}\subset \Lambda_{M}^\bot$ on the $\calY^\spl$-strata.
}    
\end{Remark}

\begin{Corollary}\label{Reducedness}
a) The moduli functor $ \mathcal{Z}^{\rm spl}(\Lambda)$ is normal, Cohen-Macaulay, reduced and of dimension $t+h$.

b) The moduli functor $\mathcal{Y}^{\rm spl}(\Lambda^{\sharp}) $ is smooth, reduced and of dimension $ n-t-h-1$.

c) The moduli functor $\calZ^\spl(\Lambda_1)\cap \mathcal{Y}^{\rm spl}(\Lambda^{\sharp}_2)$ is smooth, reduced and of dimension $ t_1-t_2-1$.
\end{Corollary}
\begin{proof}
From the local model diagram we have that every point of $ \mathcal{Z}^{\rm spl}(\Lambda)$ has an \'etale neighborhood which is also \'etale over the strata splitting model $  \M^{\spl,[2h]}_n(2t)$. Now the result follows from Theorem \ref{thm splstrata}. A similar argument works for $\mathcal{Y}^{\rm spl}(\Lambda^{\sharp}) $ and $\calZ^\spl(\Lambda_1)\cap \mathcal{Y}^{\rm spl}(\Lambda^{\sharp}_2)$.
\end{proof}

\section{Bruhat-Tits stratification}\label{BTstrat.}
In this section, we will define the Bruhat-Tits stratification of the reduced subscheme $\calN_{n, {\rm red}}^\spl$ (the \textit{reduced basic locus}) of the special fiber $\overline{\mathcal{N}}_n^{\rm spl}$. 


Let $\kappa$ be any perfect field over $\bar{k}$. Recall that $M = M(X)$ the Dieudonn\'e module of $(X, \iota, \lambda, \rho)\in \calN_n$. We denote by $T_i(M)$ (resp. $T_i(M^\sharp)$) the summation $M+\tau(M)+\cdots +\tau^i(M)$ (resp. $M^\sharp+\tau(M^\sharp)+\cdots +\tau^i(M^\sharp)$). By \cite[Proposition 2.17]{RZbook}, there exists a smallest nonnegative integer $c$ (resp. $d$) such that $T_c(M)$ (resp. $T_d(M^\sharp)$) is $\tau$-invariant. Set $\Lambda_1=T_d(M^\sharp)^\sharp\cap C$, $\Lambda_2=T_c(M)\cap C$. 

\begin{Proposition}\label{prop 51} 
We have $\Lambda_1\otimes W_{O_{F_0}}(\kappa)\subset M\subset \Lambda_2\otimes W_{O_{F_0}}(\kappa)$, and the $W_{O_{F_0}}(\kappa)$-lattices $M$ satisfy one of the following:
\begin{itemize}
    \item (Case $\calZ^\spl$)\quad $\Lambda_1\subset C$ is a vertex lattice of type $2t_1\geq 2h$ with 
    \[
    \pi M^\sharp\subset
    \pi \Lambda_1^\sharp\otimes W_{O_{F_0}}(\kappa)\subset
    \Lambda_1\otimes W_{O_{F_0}}(\kappa)\subset
    M \subset
    M^\sharp \subset
    \Lambda_1^\sharp\otimes W_{O_{F_0}}(\kappa),
    \]
 and $\Lambda_1$ is the maximal vertex lattice in $C$ such that $\Lambda_1\otimes W_{O_{F_0}}(\kappa)$ is contained in $M$.
    \item (Case $\calY^\spl$)\quad $\Lambda_2\subset C$ is a vertex lattice of type  $2t_2\leq 2h$ with
    \[
    \pi \Lambda_2^\sharp\otimes W_{O_{F_0}}(\kappa)\subset
    \pi M^\sharp\subset
    M\subset
    \Lambda_2\otimes W_{O_{F_0}}(\kappa) \subset
    \Lambda_2^\sharp\otimes W_{O_{F_0}}(\kappa) \subset
    M^\sharp,
    \]
    and $\Lambda_2$ is the minimal vertex lattice in $C$ such that $\Lambda_2\otimes W_{O_{F_0}}(\kappa)$ contains $M$.
\end{itemize}
\end{Proposition}

\begin{proof}
See \cite[Proposition 5.3]{HLS2}.
\end{proof}


Recall that $L_{\mathcal{Z}}$ (resp. $L_{\mathcal{Y}}$) denotes the set of all vertex lattices in $C$ of type $\geq 2h$ (resp. $\leq 2h$). Then we have the following:

\begin{Theorem}\label{Thm BT}
The Bruhat-Tits stratification of the reduced basic locus is
  \begin{equation}\label{BTstr}
      \calN_{n, {\rm red}}^\spl = \left( \bigcup_{\Lambda_1 \in L_{\mathcal{Z}} }\mathcal{Z}^{\rm spl}(\Lambda_1) \right) \cup \left( \bigcup_{\Lambda_2 \in L_{\mathcal{Y}}} \mathcal{Y}^{\rm spl}(\Lambda_2^\sharp) \right).
  \end{equation}
\begin{enumerate}
  \item  
   These strata satisfy the following inclusion relations:
  \begin{itemize}
    \item[(i)] For any $\Lambda_1, \Lambda_2 \in L_{\mathcal{Z}}$ of type greater than $2h$, 
    $  \Lambda_1 \subseteq \Lambda_2$ if and only if $\mathcal{Z}^{\rm spl}(\Lambda_2) \subseteq \mathcal{Z}^{\rm spl}(\Lambda_1)$.    
    \item[(ii)] For any $\Lambda_1, \Lambda_2 \in  L_{\mathcal{Y}}$ of type less than $2h$, 
    $  \Lambda_1 \subseteq \Lambda_2$ if and only if $\mathcal{Y}^{\rm spl}(\Lambda^{\sharp}_1) \subseteq \mathcal{Y}^{\rm spl}(\Lambda^{\sharp}_2)$.  

    \item[(iii)] For any $\Lambda_1\in L_{\mathcal{Z}}$ of type greater than $2h$, $\Lambda_2 \in  L_{\mathcal{Y}}$ of type less than $2h$, $  \Lambda_1 \subseteq \Lambda_2$ if and only if the intersection $\mathcal{Z}^{\rm spl}(\Lambda_1) \cap \mathcal{Y}^{\rm spl}(\Lambda_2^\sharp)$ is non-empty.
  \end{itemize}

  \item In the following, assume that $\Lambda, \Lambda'$ are vertex lattices of type $ 2t$ with $t \neq h$, and $\Lambda_0, \Lambda'_0$ are vertex lattices of type $2t$ with $ t=h$. 
  \begin{itemize}
    \item[(i)] The intersection $\mathcal{Z}^{\rm spl}(\Lambda) \cap \mathcal{Z}^{\rm spl}(\Lambda')$ (resp. $\mathcal{Y}^{\rm spl}(\Lambda^\sharp) \cap \mathcal{Y}^{\rm spl}(\Lambda^{\prime \sharp)}$) is non-empty if and only if $\Lambda'' = \Lambda + \Lambda'$ (resp. $\Lambda''=\Lambda \cap \Lambda'$) is a vertex lattice; in which case we have $\mathcal{Z}^{\rm spl}(\Lambda) \cap \mathcal{Z}^{\rm spl}(\Lambda') = \mathcal{Z}^{\rm spl}(\Lambda'')$ (resp. $\mathcal{Y}^{\rm spl}(\Lambda^\sharp) \cap \mathcal{Y}^{\rm spl}(\Lambda^{\prime \sharp}) = \mathcal{Y}^{\rm spl}(\Lambda^{\prime \prime \sharp})$).

    \item[(ii)] The intersection $\mathcal{Z}^{\rm spl}(\Lambda_0) \cap \mathcal{Z}^{\rm spl}(\Lambda_0')$ (or $\mathcal{Y}^{\rm spl}(\Lambda_0^\sharp) \cap \mathcal{Y}^{\rm spl}(\Lambda_0^{\prime \sharp})$) is always empty if $\Lambda_0 \ne \Lambda'_0$.
    
    \item[(iii)] The intersection $\mathcal{Z}^{\rm spl}(\Lambda) \cap \mathcal{Z}^{\rm spl}(\Lambda_0)$ (resp.  $\mathcal{Y}^{\rm spl}(\Lambda^\sharp) \cap \mathcal{Y}^{\rm spl}(\Lambda_0^\sharp)$) is non-empty if and only if $\Lambda \subset \Lambda_0$ (resp. $\Lambda_0 \subset \Lambda$), in which case $\mathcal{Z}^{\rm spl}(\Lambda) \cap \mathcal{Z}^{\rm spl}(\Lambda_0)$ (resp.  $\mathcal{Y}^{\rm spl}(\Lambda^\sharp) \cap \mathcal{Y}^{\rm spl}(\Lambda_0^\sharp)$) is isomorphic to $\mathbb{P}^{h+t-1}_{\bar{k}}$ (resp. $\mathbb{P}^{h-t-1}_{\bar{k}}$).
    \item[(iv)] The BT-strata $\mathcal{Z}^{\mathrm{spl}}(\Lambda_0)$ and $\mathcal{Y}^{\mathrm{spl}}(\Lambda_0^\sharp)$ are each isomorphic to the projective space $\mathbb{P}^{n-1}_{\bar{k}}$.
   \end{itemize}
\end{enumerate}
\end{Theorem}

\begin{proof}
To prove (\ref{BTstr}), it suffices to check this on $\kappa$-points (see also the proof of \cite[Theorem 3.19]{HLSY}). A point $z \in \mathcal{N}_{n, \mathrm{red}}^\spl(\kappa)$ corresponds to a pair $(M, M')$ as in Proposition \ref{DSpl}. For the remainder of the proof, we fix $\kappa$ and denote by  $\breve{\Lambda}=\Lambda\otimes W_{O_{F_0}}(\kappa)$. By taking $\Lambda_1$ or $\Lambda_2$ from Proposition \ref{prop 51}, we either have $\breve{\Lambda}_1 \subset M$ or $ M\subset \breve{\Lambda}_2$ as a unique vertex lattice of type $\geq 2h$ or $ \leq 2h$ respectively. 

If  $\Lambda_i$ has type $2h$ for $i=1,2$, then $ \breve{\Lambda}_1\subset M \subset M^\sharp \subset \breve{\Lambda}_1^\sharp$ for case $\calZ^\spl$ (resp.  $M\subset \breve{\Lambda}_2 \subset\breve{\Lambda}_2^\sharp\subset M^\sharp$ for case $\calY^\spl$), so they have to be equal. Thus, by Proposition \ref{DieudLatt}, $z \in \mathcal{Z}^{\rm spl}(\Lambda)(\kappa)$ or $z \in \mathcal{Y}^{\rm spl}(\Lambda^{\sharp})(\kappa)$ depending on $2h$.

If \( \Lambda \) is not of type $2h$, then $ M$, and so $M^{\sharp}$, is not $\tau$-invariant. By Proposition \ref{DSpl}, we have $M' \subset \tau^{-1}(M^\sharp) \cap M^\sharp$, $\operatorname{length}(M^\sharp / M') = 1$. Since $\tau^{-1}(M^\sharp)\neq M^\sharp$, we get
$\tau^{-1}(M^\sharp) \cap M^\sharp\subsetneq M^\sharp$ and so $M' = M^{\sharp} \cap \tau^{-1}(M^{\sharp})$ is uniquely determined. Since $\Lambda$ is $\tau$-invariant, we deduce that either $\breve{\Lambda}\subset M^\sharp$, which implies $\breve{\Lambda}=\tau^{-1}(\breve{\Lambda})\subset \tau^{-1}(M^\sharp)$ and so $\breve{\Lambda}\subset M'$ (Case $\calZ^\spl$), or $\breve{\Lambda}^{\sharp} \subset M^\sharp$, which implies $\breve{\Lambda}^{\sharp}=\tau^{-1}(\breve{\Lambda}^\sharp) \subset \tau^{-1}(M^\sharp)$ and so $\breve{\Lambda}^\sharp\subset M'$ (Case $\calY^\spl$). Hence, $z \in \mathcal{Z}^{\rm spl}(\Lambda)(\kappa)$ or $z \in \mathcal{Y}^{\rm spl}(\Lambda^{\sharp})(\kappa)$ by Proposition \ref{DieudLatt}. This proves  (\ref{BTstr}). 

(1). Inclusion properties (i) and (ii) follow from the definitions of the strata by Propositions \ref{DieudLatt} and \ref{prop 51}. For (iii), if the intersection $\mathcal{Z}^{\rm spl}(\Lambda_1) \cap \mathcal{Y}^{\rm spl}(\Lambda_2^\sharp)$ is non-empty and pick $(M, M')\in \mathcal{Z}^{\rm spl}(\Lambda_1) \cap \mathcal{Y}^{\rm spl}(\Lambda_2^\sharp)(\kappa)$, then $\breve{\Lambda}_1\subset M \subset M^\sharp$ and $M\subset \breve{\Lambda}_2\subset \breve{\Lambda}_2^\sharp$, thus $\breve{\Lambda}_1 \subset M\subset \breve{\Lambda}_2$. Conversely, suppose $\Lambda_1\subset \Lambda_2$, then the intersection $\M^{\spl, [2h]}_n (2t_1)\cap \M^{\spl, [2h]}_n (2t_2)$ is non-empty by \S \ref{intersection}, where $t_1$ is the type of $\Lambda_1$ and $t_2$ is the type of $\Lambda_2$. Thus, $\mathcal{Z}^{\rm loc}(\Lambda_1) \cap \mathcal{Y}^{\rm loc}(\Lambda_2^\sharp)$ is non-empty by the local model diagram (\ref{LMdiagram1}).

(2.i). For the $\mathcal{Z}^{\rm spl}$-strata, if we assume that $\Lambda''$ is a vertex lattice, then we can easily see that $\mathcal{Z}^{\rm spl}(\Lambda) \cap \mathcal{Z}^{\rm spl}(\Lambda') = \mathcal{Z}^{\rm spl}(\Lambda'')$ by construction. On the other hand, if we assume that $\mathcal{Z}^{\rm spl}(\Lambda) \cap \mathcal{Z}^{\rm spl}(\Lambda')$ is nonempty and pick $ (M,M') \in  \mathcal{Z}^{\rm spl}(\Lambda) \cap \mathcal{Z}^{\rm spl}(\Lambda')(\kappa)$, then $\Lambda_1 \supset \Lambda + \Lambda'$ where $ \Lambda_1$ is the maximal vertex lattice contained in $M$ from Proposition \ref{prop 51}. Then $\Lambda + \Lambda' \subset \Lambda_1 \subset \Lambda_1^\sharp \subset \Lambda^\sharp \cap (\Lambda')^\sharp = (\Lambda + \Lambda')^\sharp$. Similarly, we have $\pi(\Lambda+\Lambda')^\sharp\subset (\Lambda+\Lambda')$. Hence, $\Lambda''=\Lambda + \Lambda'$ is a vertex lattice. For $\mathcal{Y}^{\rm spl}$-strata, note that $ (M,M') \in  \mathcal{Y}^{\rm spl}(\Lambda^{\sharp}) \cap \mathcal{Y}^{\rm spl}((\Lambda')^{\sharp})(\kappa)$ gives  $\Lambda_2\subset \Lambda\cap \Lambda'\subset (\Lambda^\sharp+\Lambda')^\sharp\subset \Lambda_2^\sharp$ by the minimality of $\Lambda_2$ from Proposition \ref{prop 51}. Then, arguing as in the case of the $\mathcal{Z}^{\rm spl}$-strata, we obtain the desired result.

(2.ii). The statement follows from Proposition \ref{DieudLatt} (see also Corollary \ref{t0}).

(2.iii). For $\mathcal{Z}^{\rm spl}$-strata, a point $(M, M') \in  \mathcal{N}^{\rm spl}_{n}(\kappa) $ is in $\mathcal{Z}^{\rm spl}(\Lambda) \cap \mathcal{Z}^{\rm spl}(\Lambda_0)$ if and only if $M = \breve{\Lambda}_0$ and $\breve{\Lambda} \subset M \subset M^\sharp\subset \breve{\Lambda}^\sharp$, $\breve{\Lambda} \subset M' \subset M^\sharp$ by Propositions \ref{DieudLatt} and \ref{prop 51}. This shows that $\Lambda \subset \Lambda_0$, and $M'$ corresponds to a point in $\mathbb{P}(\Lambda_0^\sharp / \Lambda)(\kappa)$. Note that $\Lambda_0$ (resp. $\Lambda$) is a vertex lattice of type $2h$ (resp. $2t$). So $\dim(\Lambda_0^\sharp / \Lambda)=t+h$. Thus, we can see that $ \mathcal{Z}^{\rm spl}(\Lambda) \cap \mathcal{Z}^{\rm spl}(\Lambda_0)(\kappa) =\mathbb{P}^{h+t-1}_{\kappa}$. 
Similarly, for $\mathcal{Y}^{\rm spl}$-strata, we have $M = \breve{\Lambda}_0 $ and $M\subset \breve{\Lambda} \subset \breve{\Lambda}^\sharp\subset M^\sharp$, $\breve{\Lambda}^\sharp \subset M' \subset M^\sharp$. Thus, $\Lambda_0\subset \Lambda$ and 
$M'$ corresponds to a point in $\mathbb{P}(\Lambda_0^\sharp / \Lambda^\sharp)(\kappa)\simeq \mathbb{P}^{h-t-1}(\kappa)$.

(2.iv). This claim follows from Corollary \ref{t0}.
\end{proof}

\section{Global properties of Bruhat-Tits strata}\label{GlobalPropertiesSection}
The goal of this section is to prove that the BT-strata in the RZ space $\mathcal{N}^{\rm spl}_n$ are connected and irreducible. To accomplish that, we will identify the BT-strata with certain (modified) Deligne-Lusztig varieties.

\subsection{Deligne-Lusztig varieties}\label{DLVs}
In this section, we consider a class of (modified) Deligne–Lusztig varieties arising from symplectic and orthogonal groups. 

\subsubsection{Symplectic Case}\label{symsec} Assume that the lattice $\Lambda \subset C$ is of type $2t$ with $ t>h$ and consider the $k$-vector space $V_{\Lambda} = \Lambda^{\sharp}/\Lambda$ of dimension $2t$ with induced symplectic form $ \langle \,, \, \rangle  $. Define $  V_{\Lambda,\bar{k}} := V_{\Lambda} \otimes_k \bar{k} $, denote by $\Phi$ its Frobenius endomorphism and denote the bilinear extension of $ \langle \,, \, \rangle$ to $ V_{\Lambda,\bar{k}}$ still by $ \langle \,, \, \rangle$. Let $G$ be the special symplectic group ${\rm Sp}(V_{\Lambda,\bar{k}})$. We fix a maximal torus and Borel subgroup $T\subset B\subset G$ which is stable under the $\Phi$-action. Let \( \mathrm{Gr}(i, V_{\Lambda,\bar{k}}) \) be the Grassmannian variety parametrizing rank \( i \) locally direct summands of $ V_{\Lambda,\bar{k}}$. Consider the parabolic subgroup $P\subset G$, where $G/P$ parametrizes isotropic subspaces in $V_{\Lambda,\bar{k}}$ of dimension $t-h$. Denote by $\mathrm{SGr}(i, V_{\Lambda,\bar{k}})=G/P$. The $\bar{k}$-points are 
$ \mathrm{SGr}(i, V_{\Lambda,\bar{k}})(\bar{k})=\left\{ U \in \mathrm{Gr}(i, V_{\Lambda,\bar{k}})(\bar{k}) \ \middle| \  \langle U , U \rangle =0 \right\}$. 

Consider the subvariety $S_{\Lambda}$ of $ \mathrm{SGr}(t-h, V_{\Lambda,\bar{k}})$ (see \cite[\S 6.2]{HLS2}) given by 
\[
S_{\Lambda}(\bar{k}) = \left\{ U \in \mathrm{SGr}(t-h, V_{\Lambda,\bar{k}})(\bar{k}) \ \middle| \  \operatorname{dim} (U \cap \Phi(U)) \geq t-h-1 \right\}
\]
which is stratified by certain (generalized) Deligne-Lusztig varieties. 
We refer to \cite[\S 6.1]{HLS2} for more details. The variety $S_\Lambda$ is irreducible and of dimension $t+h$ (see \cite[Theorem 6.3]{HLS2}). 
Next, let $S'_{\Lambda}$ be the reduced closed subscheme of $\mathrm{SGr}(t-h, V_{\Lambda,\bar{k}}) \times \mathrm{Gr}(t+h-1, V_{\Lambda,\bar{k}})$ whose $\bar{k}$-points are specified by
\[
S'_{\Lambda}(\bar{k}) = \left\{ (U, U') \in \left(\mathrm{SGr}(t-h, V_{\Lambda,\bar{k}}) \times \mathrm{Gr}(t+h-1, V_{\Lambda,\bar{k}})\right)(\bar{k}) \;\middle|\; U' \subset U^{\sharp} \cap \Phi(U^{\sharp}) \right\}.
\]
(Here $U^{\sharp}$ is the dual of $U$ with respect to the symplectic form $ \langle \,, \, \rangle  $ of $V_{\Lambda} $). Then the variety $S'_{\Lambda}$ is a projective subvariety of $\mathrm{SGr}(t-h, V_{\Lambda,\bar{k}}) \times \mathrm{Gr}(t+h-1, V_{\Lambda,\bar{k}})$. Note that by \cite[Lemma 4.4]{HZ}, we have $[U^\sharp : U^\sharp \cap \Phi(U^\sharp)] = [U : U \cap \Phi(U)]$, and thus the conditions $U' \subset U^\sharp \cap \Phi(U^\sharp)$ and $U' \mathrel{\overset{\scriptstyle \leq 1}{\subset}} U^\sharp$ imply that $[U : U \cap \Phi(U)] \leq 1$. Hence, there is a forgetful surjective map $\varphi_{\Lambda}: S'_{\Lambda} \rightarrow S_{\Lambda}  $ given by $ (U, U') \mapsto U$.

\begin{Lemma}\label{lm 61}
The morphism $\varphi_{\Lambda}$ is a projective morphism. It is an isomorphism outside the closed subscheme $ T = \{U\in S_{\Lambda} \, | \, U = \Phi(U) \}$ of $S_{\Lambda}$. For a point $y\in T$ we have $\varphi_{\Lambda}^{-1}(y) \cong \mathbb{P}^{t+h-1}_{\bar{k}}$.
\end{Lemma}
\begin{proof}
First, we know that $\varphi_{\Lambda}$ is projective as it is a morphism between projective schemes. The subscheme $T$ is closed by \cite[\S 6.2]{HLS2}. Consider a $\bar{k}$-point $U \in S_{\Lambda}(\overline{\kappa})/ T$, then $U \cap \Phi(U)$ has dimension $t+h-1$. The fiber of $U$ under $\varphi_{\Lambda}$ contains pairs $(U,U')$ such that $ U' \subset U^{\sharp} \cap \Phi(U^{\sharp}) $ and $U'$ has dimension $t+h-1$. Hence, $U'$ is uniquely determined and equals $ U^{\sharp} \cap \Phi(U^{\sharp}) $. Now, assume that $ U \in T$. Then $U = \Phi(U)$ and $U'$ can be any element in $\mathrm{Gr}(t+h-1, U^{\sharp}) \cong \mathbb{P}^{t+h-1}_{\bar{k}}$. This finishes the proof of the lemma.  
\end{proof}
From the above, we can deduce that

\begin{Corollary}\label{Cor 62}
The projective scheme $ S'_{\Lambda}$ has dimension $t+h$.    
\end{Corollary}
\begin{proof}
Set $ \calU_1 = S_{\Lambda}\setminus T$ and $\calU_2:=\varphi_{\Lambda}^{-1}(\calU_1)\simeq \calU_1$. Let $X_0$ be the unique irreducible component of $S'_{\Lambda}$ that contains $\calU_2$. The open subscheme $\calU_1$ is dense in the irreducible variety $S_{\Lambda}$. Thus, $\dim X_0=t+h$. Now, assume $X_1\neq X_0$ is another irreducible component of $S'_{\Lambda}$. Using Lemma \ref{lm 61} and the fact that $T$ is zero dimensional (see \cite[\S 6.2]{HLS2}) we have that $\varphi_{\Lambda}(X_1)=t\in T$ where $t$ is a closed point of $T$ and $\displaystyle \dim X_1\le\dim\varphi_{\Lambda}^{-1}(t)=t+h-1.$ Therefore, we conclude that  $\dim S'_{\Lambda}=t+h$.
\end{proof}
As will be shown in Proposition \ref{S'irreducibility}, $ S'_{\Lambda}$ is also irreducible.
\subsubsection{Orthogonal Case} 
Assume that the lattice $\Lambda \subset C$ is of type $t <h$ and consider the $k$-vector space $V_{\Lambda^{\sharp}} = (\pi^{-1}\Lambda) /\Lambda^{\sharp}$ with induced orthogonal form $ ( \,, \, )  $. Define $ V_{\Lambda^{\sharp},\bar{k}} := V_{\Lambda^{\sharp}} \otimes_k \bar{k} $, 
denote the bilinear extension of $ ( \,, \, )$ to $ V_{\Lambda^{\sharp},\bar{k}}$ still by $( \,, \, )$. Let $G$ be the special orthogonal group ${\rm SO}(V_{\Lambda^{\sharp},\bar{k}})$. Let \( \mathrm{Gr}(i, V_{\Lambda^{\sharp},\bar{k}}) \) be the Grassmannian variety 
and let $ \mathrm{OGr}(i, V_{\Lambda^{\sharp},\bar{k}})$ be the subvariety of $ \mathrm{Gr}(i, V_{\Lambda^{\sharp},\bar{k}})$ given by $\mathrm{OGr}(i, V_{\Lambda^{\sharp},\bar{k}}) = \left\{ U \in \mathrm{Gr}(i, V_{\Lambda^{\sharp},\bar{k}}) \ \middle| \  ( U , U ) =0 \right\}$. Similar to \S \ref{symsec}, we consider the reduced closed subvariety $R_{\Lambda^{\sharp}}$ of $ \mathrm{OGr}(h-t, V_{\Lambda^{\sharp},\bar{k}})$, where the $\bar{k}$-points are
\[
R_{\Lambda^{\sharp}}(\bar{k}) = \left\{ U \in \mathrm{OGr}(h-t, V_{\Lambda^{\sharp},\bar{k}})(\bar{k}) \ \middle| \  \operatorname{dim} (U \cap \Phi(U)) \geq h-t-1 \right\},
\]
(see \cite[\S 6.3]{HLS2}). The variety $R_{\Lambda^{\sharp}}$ is irreducible of dimension $n-t-h-1$ and admits a stratification by (generalized) Deligne-Lusztig varieties (see \cite[Theorem 6.10]{HLS2}. 
Next, define $R'_{\Lambda^{\sharp}}$ to be the subvariety of $\mathrm{OGr}(h-t, V_{\Lambda^{\sharp},\bar{k}}) \times \mathrm{OGr}(h-t-1, V_{\Lambda^{\sharp},\bar{k}})$ whose $\bar{k}$-points are specified by
\[
R'_{\Lambda^{\sharp}}(\bar{k}) = \left\{ (U, U') \in \left(\mathrm{OGr}(h-t, V_{\Lambda^{\sharp},\bar{k}}) \times \mathrm{OGr}(h-t-1, V_{\Lambda^{\sharp},\bar{k}})\right)(\bar{k}) \;\middle|\; U' \subset U \cap \Phi(U) \right\}.
\]
The variety $R'_{\Lambda^{\sharp}}$ is a projective subvariety and 
we have the forgetful map $\varphi_{\Lambda^{\sharp}}: R'_{\Lambda^{\sharp}}\rightarrow R_{\Lambda^{\sharp}}  $ given by $ (U, U') \mapsto U$. 
\begin{Proposition}\label{Rsmooth}
The projective variety $R'_{\Lambda^{\sharp}}$ is irreducible and smooth of dimension $ n-t-h-1$.    
\end{Proposition}
\begin{proof}
Set $i=h-t$. Define $\mathrm{OGr}(i, i-1)$ to be the subvariety of $\mathrm{OGr}(i, V_{\Lambda^{\sharp},\bar{k}}) \times \mathrm{OGr}(i-1, V_{\Lambda^{\sharp},\bar{k}})$ whose $\bar{k}$ points are specified by
\[
\mathrm{OGr}(i, i-1)(\bar{k})=\{(U, U') \in \left(\mathrm{OGr}(i, V_{\Lambda^{\sharp},\bar{k}}) \times \mathrm{OGr}(i-1, V_{\Lambda^{\sharp},\bar{k}})\right)(\bar{k})\mid U'\subset U\}.
\]
Consider the following closed immersions:
\[
f_1: \mathrm{OGr}(i, i-1)^2\rightarrow (\mathrm{OGr}(i, V_{\Lambda^{\sharp},\bar{k}}) \times \mathrm{OGr}(i-1, V_{\Lambda^{\sharp},\bar{k}}))^2
\]
given by $(U_1,U_1',U_2, U_2')\mapsto (U_1,U_1',U_2, U_2')$ and
\[
f_2: \mathrm{OGr}(i, V_{\Lambda^{\sharp},\bar{k}}) \times \mathrm{OGr}(i-1, V_{\Lambda^{\sharp},\bar{k}})\rightarrow(\mathrm{OGr}(i, V_{\Lambda^{\sharp},\bar{k}}) \times \mathrm{OGr}(i-1, V_{\Lambda^{\sharp},\bar{k}}))^2
\]
given by $(U,U')\mapsto (U,U',\Phi(U),U')$. By construction, $R'_{\Lambda^{\sharp}}=\mathrm{Im}(f_1)\cap \mathrm{Im}(f_2)$. Since $\mathrm{OGr}(i, i-1)$ and $\mathrm{OGr}(i, V_{\Lambda^{\sharp},\bar{k}}) \times \mathrm{OGr}(i-1, V_{\Lambda^{\sharp},\bar{k}})$ are homogeneous varieties, they are smooth. The Frobenius $\Phi$ induces the zero map on the tangent space and as in the proof of \cite[Proposition 3.2]{HLSY} we deduce that the intersection is transversal. Hence, $R'_{\Lambda^{\sharp}}$ is smooth. Similar to Lemma \ref{lm 61}, the morphism $\varphi_{\Lambda^\sharp}$ is an isomorphism outside the closed subvariety $T'= \{U\in R_{\Lambda^\sharp} \mid U = \Phi(U) \}$. Since $R_{\Lambda^{\sharp}}$ is irreducible of dimension $n-h-t-1$, we get that $R_{\Lambda^\sharp}'$ is irreducible with the same dimension. This finishes the proof of the proposition.
\end{proof}

\subsection{Relation of Deligne-Lusztig varieties with Bruhat-Tits strata}
\subsubsection{$\calZ^\spl$-strata}
Let $\Lambda \subset C$ be a vertex lattice of type $2t$ with $t> h$. The goal is to construct an isomorphism of the BT-stratum $\mathcal{Z}^{\rm spl}(\Lambda)$ and the modified Deligne-Lusztig variety $ S'_{\Lambda}$ defined in \ref{DLVs}. 

For any $\bar{k}$-algebra $R$ and an $R$-point $(X, \iota, \lambda, \rho, \mathcal{F}) \in \mathcal{Z}^{\rm spl}(\Lambda)(R)$, we have the following chains of isogenies
\[
\rho_{\Lambda,\Lambda^\sharp} : X_{\Lambda, R}
\xrightarrow{\rho_{\Lambda, X}} X
\xrightarrow{\lambda} X^\vee
\xrightarrow{\rho_{X^\vee, \Lambda^\sharp}} X_{\Lambda^\sharp, R}.
\]
Applying de Rham realization, we obtain the sequence of
$R$-modules:
\begin{equation}
D(X_{\Lambda, R}) \xrightarrow{D(\rho_{\Lambda, X})}
D(X) \xrightarrow{D(\lambda)} 
D(X^{\vee})\xrightarrow{D(\rho_{ X^\vee, \Lambda^\sharp})} 
D(X_{\Lambda^{\sharp}, R}).
\end{equation}
\quash{
\[
\begin{tikzcd}
\Lambda_{-t,R}\arrow{r} \arrow[d,"\sim"]
&
\Lambda_{-h,R} \arrow{r} \arrow[d,"\sim"]
&
\Lambda_{h,R} \arrow{r} \arrow[d,"\sim"]
&
\Lambda_{t,R} \arrow[d,"\sim"]
\\
D(X_{\Lambda, R}) \arrow[r, "D(\rho_{\Lambda, X})"]
&
D(X) \arrow{r} \arrow[r, "D(\lambda)"] 
&
D(X^{\vee})\arrow[r, "D(\rho_{ X^\vee, \Lambda^\sharp})"]
&
D(X_{\Lambda^{\sharp}, R})
\end{tikzcd}.
\]
}
Set $D(\rho_{\Lambda, \Lambda^\sharp})=D(\rho_{ X^\vee, \Lambda^\sharp})\circ D(\lambda)\circ D(\rho_{\Lambda, X})$.
By definition, the image $\mathrm{Im}(D(\rho_{\Lambda, \Lambda^\sharp}))$ is a locally free direct summand of $D(X_{\Lambda^\sharp, R})$ of corank $2t$, such that
\[
D(X_{\Lambda^\sharp, R}) / \mathrm{Im}(D(\rho_{\Lambda, \Lambda^\sharp})) \simeq \Lambda^\sharp / \Lambda \otimes_{\bar{k}} R = V_{\Lambda, R}.
\]
It is easy to see that we have a symplectic form $\bb ~ ,~ \pp$ on $V_{\Lambda, R}$ given by $\bb x,y\pp=\pi h(\tilde{x},\tilde{y})$, where $\tilde{x}$ and $\tilde{y}$ are lifting points of $x, y$ in $\Lambda^\sharp$.

Since $\Lambda$ is a vertex lattice, we have $\ker(\rho_{\Lambda, \Lambda^\sharp}) \subset X_\Lambda[\iota(\pi)]$. This implies that the kernel of the
composition
\[
\rho_{X, \Lambda^\sharp} := \rho_{X^\vee, \Lambda^\sharp} \circ \lambda : X \to X_{\Lambda^\sharp, R}
\]
lies in $X[\iota(\pi)]$. Therefore, there exists an isogeny $\tilde{\rho}_{X, \Lambda^\sharp} : X_{\Lambda^\sharp, R} \to X$ such that
\[
\tilde{\rho}_{X, \Lambda^\sharp} \circ \rho_{X, \Lambda^\sharp} = \iota(\pi) : X \to X.
\]

Recall that $\operatorname{Fil}(X)\subset D(X)$ is the Hodge filtration, $\operatorname{Fil}^0(X)\subset  \operatorname{Fil}(X)$ is a locally direct summand of rank $1$ and $\mathcal{F} \subset  \operatorname{Lie}(X^\vee)$ is the perpendicular complement of $ \operatorname{Fil}^0(X)$ under the perfect pairing (\ref{perf.pair.}) of rank $n-1$. We have the following diagram:
\begin{equation}
\begin{tikzcd}
D(X_{\Lambda,R}) \arrow[r,"D(\rho_{\Lambda, X^\vee})"] 
&
D(X^\vee) \arrow[r,"D(\rho_{X^\vee, \Lambda^\sharp})"] 
&
D(X_{\Lambda^{\sharp},R}) \arrow[r,"D(\tilde{\rho}_{X, \Lambda^\sharp})"] 
&
D(X)
\\
&
\mathrm{Pr}^{-1}(\calF)\arrow[u,hook]
&&
\operatorname{Fil}(X) \arrow[u,hook]
\end{tikzcd},    
\end{equation}
where $ \operatorname{Pr}: D(X^\vee) \rightarrow \operatorname{Lie}(X^\vee)$ is the natural quotient homomorphism of $R$-modules. 
\begin{Lemma}\label{lm 64}
For any $\bar{k}$-algebra $R$ and an $R$-point $(X, \iota, \lambda, \rho, \mathcal{F}) \in \mathcal{Z}^{\rm spl}(\Lambda)(R)$, the preimage $D(\tilde{\rho}_{X, \Lambda^\sharp})^{-1}(\mathrm{Fil}(X))$ (resp. the image $D(\rho_{X^\vee, \Lambda^\sharp})(\mathrm{Pr}^{-1}(\calF))$ ) is a locally free direct summand of $D(X_{\Lambda^\sharp, R})$ that contains
$\mathrm{Im}(D(\rho_{\Lambda, \Lambda^\sharp}))$. Moreover, the quotients
\[
\begin{array}{c}
 U(X) := D(\tilde{\rho}_{X, \Lambda^\sharp})^{-1}(\mathrm{Fil}(X)) / \mathrm{Im}(D(\rho_{\Lambda, \Lambda^\sharp})),  \\
U'(X) := D(\rho_{X^\vee,\Lambda^{\sharp}})(\operatorname{Pr}^{-1}(\mathcal{F}))/\mathrm{Im}(D(\rho_{\Lambda, \Lambda^\sharp})).
\end{array}
\]
are locally free direct summands of $V_{\Lambda, R}$ of ranks $t - h$  and $t+h-1$ respectively.
\end{Lemma}


\begin{proof}
It is sufficient to check the condition on $\bar{k}$-points of $\mathcal{Z}^{\rm spl}(\Lambda)$. The proof of the $U(X)$ part can be found in \cite[Lemma 7.1]{HLS2}. Note that $U(X)$ is isomorphic to $\Phi^{-1}(M(X))/\breve{\Lambda}$ where $\breve{\Lambda}:=\Lambda\otimes W_{O_{F_0}}(\bar{k})$. 
For the $U'(X)$ part, consider the chain of Dieudonn\'e lattices
\[
\breve{\Lambda}\subset M'(X)\subset M(X)^\sharp \subset \breve{\Lambda}^\sharp,
\]
corresponding to a point $(X, \iota, \lambda, \rho, \mathcal{F}) \in \mathcal{Z}^{\rm spl}(\Lambda)(\bar{k})$. By Proposition \ref{DieudLatt}, we deduce that $\mathrm{Im}(D(\rho_{\Lambda, \Lambda^\sharp})) \subset D(\rho_{X^\vee,\Lambda^{\sharp}})(\operatorname{Pr}^{-1}(\mathcal{F}))$ and 
\[
U'(X)=D(\rho_{X^\vee,\Lambda^{\sharp}})(\operatorname{Pr}^{-1}(\mathcal{F}))/\mathrm{Im}(D(\rho_{\Lambda, \Lambda^\sharp})\simeq M'(X)/\breve{\Lambda}\subset M(X)^\sharp/\breve{\Lambda}.
\]
Note that $M(X)^\sharp/\breve{\Lambda}$ is of dimension $t+h$. Thus, $U'(X)$ is a locally free direct summand of $V_{\Lambda, R}$ of rank $t+h-1$. This finishes the proof of the lemma.
\end{proof}

\begin{Proposition}\label{BijZstrata}
Let $ \kappa$ be a perfect field over $\bar{k}$. There exists a bijective morphism $ f_{\calZ}: \mathcal{Z}^{\rm spl}(\Lambda)(\kappa) \longrightarrow S_{\Lambda}'(\kappa)$ given by $(X, \iota, \lambda, \rho, \calF) \mapsto (U(X), U'(X))$. 
\end{Proposition}


\begin{proof}
By Dieudonn\'e theory, a point $z\in \mathcal{Z}^{\rm spl}(\kappa)$ corresponds to a pair of lattices $(M, M')$ satisfying
\[
\breve{\Lambda}\subset M\subset M^\sharp \subset \breve{\Lambda}^\sharp, \quad  \breve{\Lambda}\subset M'\subset M^\sharp \subset \breve{\Lambda}^\sharp.
\]
Similar to Lemma \ref{lm 64}, we can show that 
\[
U(X)\simeq \Phi^{-1}(M)/\breve{\Lambda}, \quad
U'(X)\simeq M'/\breve{\Lambda}.
\]
Note that $U(X)$ is contained in the dual lattice $U(X)^\sharp=\Phi^{-1}(M)^\sharp/\breve{\Lambda}$. Thus, the pair $(U(X), U'(X))$ belongs to $\left(\mathrm{SGr}(t-h, V_{\Lambda,\bar{k}}) \times \mathrm{Gr}(t+h-1, V_{\Lambda,\bar{k}})\right)(\kappa)$. The relation $M' \subset \Phi^{-1}(M^\sharp) \cap M^\sharp$ in Proposition \ref{DSpl} is equivalent to 
\[
U'(X)\subset U(X)^\sharp\cap \Phi(U(X))^\sharp.
\]
Therefore, $f_{\calZ}(z)\in S_{\Lambda}'(\kappa)$.

Conversely, assume $(U, U') \in S'_{\Lambda}(\kappa)$ and let $M = \mathrm{Pr}^{-1}(\Phi(U))$ and $M' = \mathrm{Pr}^{-1}(U')$, where $\mathrm{Pr} : \breve{\Lambda}^\sharp \to \breve{\Lambda}^\sharp/\breve{\Lambda}$ is the natural quotient map. Then, by definition we have $\breve{\Lambda}\subset M $, $\breve{\Lambda} \subset M'$ and $M'\subset \tau^{-1}(M^\sharp)\cap M^\sharp$. To show that $(M, M') \in \mathcal{Z}^{\mathrm{spl}}(\Lambda)(\kappa)$, it suffices to show that 
\[
V M^{\sharp} \subset M',\quad \Pi M^\sharp\subset M,\quad \Pi M\subset \tau^{-1}(M)\subset \Pi^{-1}M.
\]
Observe that $V M^{\sharp} \subset V \breve{\Lambda}^\sharp = \Pi \breve{\Lambda}^\sharp \subset \breve{\Lambda} \subset M'$ and so $V M^{\sharp} \subset M'$. Similarly, $\Pi M^\sharp\subset \Pi \breve{\Lambda}^\sharp\subset \breve{\Lambda}\subset M$, $\Pi M\subset \Pi \breve{\Lambda}^\sharp\subset \breve{\Lambda}\subset \tau^{-1}M$ and $\tau^{-1}M\subset \breve{\Lambda}^\sharp\subset \Pi^{-1}\breve{\Lambda}\subset \Pi^{-1}M$. This shows that $(M, M ')$ satisfies the conditions in Propositions \ref{DSpl} and \ref{DieudLatt}. Hence, $f_{\mathcal{Z}}$ defines a bijection
between $\mathcal{Z}^{\rm spl}(\Lambda)(\kappa)$ and $S'_{\Lambda}(\kappa)$.
\end{proof}

\begin{Theorem}\label{Isom.X}
The map $f_{\mathcal{Z}}:  \mathcal{Z}^{\rm spl}(\Lambda)\rightarrow S'_{\Lambda} $ is a closed immersion.
\end{Theorem}
\begin{proof}
From Lemma \ref{BijZstrata}, we know that $f_{\mathcal{Z}}$ is a bijection for any perfect field $\kappa$ over $\bar{k}$. Moreover, as in \cite[Proposition 7.5]{HLS2}, using the theory of displays, we can show—by the same proof as above—that this bijection extends to any field $\kappa'$ over $\bar{k}$. 
In particular, we obtain that $f_\calZ$ is a monomorphism. Note that $f_{\mathcal{Z}}$ is proper as a morphism between projective varieties. From the above we deduce that $f_{\mathcal{Z}}$ is a closed immersion. 
\end{proof}

\begin{Corollary}
The BT-stratum $ \mathcal{Z}^{\rm spl}(\Lambda)$ is irreducible.    
\end{Corollary}
\begin{proof}
Recall from Corollary \ref{Reducedness} that $\mathcal{Z}^{\rm spl}(\Lambda)$ is normal, Cohen-Macaulay 
and of dimension $t+h$. Combining this with the above theorem and with the fact that $S'_{\Lambda} $ has a unique irreducible component of dimension $t+h$ (see Corollary \ref{Cor 62}) the irreducibility of $ \mathcal{Z}^{\rm spl}(\Lambda)$ follows. 
\end{proof}

\begin{Proposition}\label{S'irreducibility}
     $S'_{\Lambda}$ is irreducible.
\end{Proposition}
\begin{proof}
By the proof of Corollary \ref{S'irreducibility}, $S'_{\Lambda}$ has a unique component $X_0$ of dimension $t+h$.
The closed immersion $f_{\mathcal{Z}}:  \mathcal{Z}^{\rm spl}(\Lambda)\rightarrow S'_{\Lambda} $ is bijective on geometric points and $\dim Z^{\mathrm{spl}}(\Lambda)=t+h$, hence $\operatorname{Im}(f_{\mathcal{Z}})=X_0$. If $X_1\neq X_0$ were another component, then for any $t\in T$ the fiber
$\varphi_\Lambda^{-1}(t)$ is irreducible and equals $\mathbb{P}^{t+h-1}_{\bar{k}}$ (see Lemma \ref{lm 61}).
Since it meets $X_0$, irreducibility forces $\varphi_\Lambda^{-1}(t)\subset X_0$,
contradicting $X_1\neq X_0$. Thus $S'_{\Lambda}$ is irreducible.
\end{proof}

\begin{Remark}
{\rm
From \cite[Theorem 7.3]{HLS2}, there exists an isomorphism $\Phi_{\calZ}: \calZ^\loc(\Lambda)\rightarrow S_{\Lambda}$ given by $(X, \iota, \lambda, \rho) \mapsto U(X)$. It is easy to see that we have the following commutative diagram:
\[
\begin{tikzcd}
\calZ^\spl(\Lambda) \arrow[r,"f_\calZ"]\arrow[d,"\mathrm{Pr}_1"]
&
S_{\Lambda}'\arrow[d,"\mathrm{Pr}_2"]
\\
\calZ^\loc(\Lambda)\arrow[r,"\Phi_\calZ"]
&
S_\Lambda
\end{tikzcd},
\]
where $\mathrm{Pr}_1$ is given by $(X, \iota, \lambda, \rho, \calF)\mapsto (X, \iota, \lambda, \rho)$ and $\mathrm{Pr}_2$ is given by $(U, U')\mapsto U$.

}
\end{Remark}
\quash{
\begin{Lemma}
The map $g_{\mathcal{Z}}$ defines a bijection between $\mathcal{Z}^{\rm spl}(\Lambda)(\kappa)$ and $S'_{\Lambda}(\kappa)$ where $\kappa$ is a perfect field over $\bar{k}$.
\end{Lemma}
\begin{proof}
A point $x \in \mathcal{Z}^{\rm spl}(\Lambda)(\kappa)$ corresponds to a pair $(M, M')$ as in Proposition \ref{DieudLatt}. By the definition of
$g_{\mathcal{Z}}$ we have $g_{\mathcal{Z}}(x) = (U, U')$ where
\[
(U, U') = (\Pi^{-1} V M / \breve{\Lambda}, M' / \breve{\Lambda}) = (\tau^{-1}(M)/\breve{\Lambda}, M' / \breve{\Lambda}) = (\Phi^{-1}(M/\breve{\Lambda}), M'/\breve{\Lambda}).
\]
From (a) of Lemma \ref{bijectionLoc} and its proof to show that $g_{\mathcal{Z}}(x) \in S'_{\Lambda}(\kappa)$ it suffices to prove that $U' \subset \Phi(U^{\sharp})\cap U^{\sharp} $. Recall from Proposition \ref{DSpl} that we have the relation $M' \subset \tau^{-1}(M^\sharp) \cap M^\sharp$. Now observe that conditions $M' \subset M^{\sharp} $ and $M'\subset \tau^{-1}(M^{\sharp}) $ are equivalent to $U' \subset \Phi(U^{\sharp}) $ and $U' \subset U^{\sharp} $ respectively. This shows that $g_{\mathcal{Z}}(x) \in S'_{\Lambda}(\kappa)$.
\end{proof}
}

\subsubsection{$\calY^\spl$-strata}
Let $\Lambda \subset C$ be a vertex lattice of type $2t$ with $t< h$. The goal of this section is to construct an isomorphism between the BT-stratum $\mathcal{Y}^{\mathrm{spl}}(\Lambda^{\sharp})$ and the modified Deligne–Lusztig variety $R'_{\Lambda^{\sharp}}$ defined in \ref{DLVs}. Since the construction is similar to that of the previous section, our discussion will be brief. 

For any $\bar{k}$-algebra $R$ and an $R$-point $(X, \iota, \lambda, \rho, \mathcal{F}) \in \mathcal{Y}^{\rm spl}(\Lambda^\sharp)(R)$, we have the following chains of isogenies
\[
\rho_{\Lambda^\sharp,\pi^{-1}\Lambda} : X_{\Lambda^\sharp, R}
\xrightarrow{\rho_{\Lambda^\sharp, X^\vee}} X^\vee
\xrightarrow{\lambda^\vee} X
\xrightarrow{\rho_{X^, \pi^{-1}\Lambda}} X_{\pi^{-1}\Lambda, R}.
\]
Applying de Rham realization, we obtain the sequence of $R$-modules:
\begin{equation}
D(X_{\Lambda^\sharp, R}) \xrightarrow{D(\rho_{\Lambda^\sharp, X^\vee})}
D(X^\vee) \xrightarrow{D(\lambda^\vee)} 
D(X)\xrightarrow{D(\rho_{X^, \pi^{-1}\Lambda})} 
D(X_{\pi^{-1}\Lambda, R}).
\end{equation}
\quash{
\[
\begin{tikzcd}
\Lambda_{t,R}\arrow{r} \arrow[d,"\sim"]
&
\Lambda_{h,R} \arrow{r} \arrow[d,"\sim"]
&
\Lambda_{n-h,R} \arrow{r} \arrow[d,"\sim"]
&
\Lambda_{n-t,R} \arrow[d,"\sim"]
\\
D(X_{\Lambda^\sharp, R}) \arrow[r]
&
D(X^\vee) \arrow{r} \arrow[r] 
&
D(X)\arrow[r]
&
D(X_{\pi^{-1}\Lambda, R})
\end{tikzcd}
\]
}
By definition, the image $\mathrm{Im}(D(\rho_{\Lambda^\sharp,\pi^{-1}\Lambda}))$ is a locally free direct summand of $D(X_{\pi^{-1}\Lambda, R})$ of corank $n-2t$, such that
\[
D(X_{\pi^{-1}\Lambda, R}) / \mathrm{Im}(D(\rho_{\Lambda^\sharp,\pi^{-1}\Lambda})) \simeq (\pi^{-1}\Lambda) / \Lambda^\sharp \otimes_{\bar{k}} R =     V_{\Lambda^\sharp, R}.
\]
We have a symmetric form $(  ~ ,~ )$ on $V_{\Lambda^\sharp, R}$ given by $(x,y)=\pi_0h(\tilde{x},\tilde{y})$ where $\tilde{x}, \tilde{y}$ are the lifting points of $x, y$ in $\pi^{-1}\Lambda$.

We have $\ker[\rho_{X^\vee, \pi^{-1}\Lambda}] \subset X^\vee[\pi]$, since $\Lambda$ is a vertex lattice, and so there exists an isogeny $\tilde{\rho}_{X^\vee, \pi^{-1}\Lambda} : X_{\pi^{-1}\Lambda} \to X^\vee$ such that $
\tilde{\rho}_{X^\vee, \pi^{-1}\Lambda} \circ \rho_{X^\vee, \pi^{-1}\Lambda} = \iota(\pi) : X^\vee \to X^\vee$. Consider the following diagram
\[
\begin{tikzcd}[column sep=huge]
D(X_{\Lambda^\sharp,R}) \arrow[r,"D(\rho_{\Lambda^\sharp, X^\vee})"] 
&
D(X^\vee) \arrow[r,"D(\rho_{X^\vee, \pi^{-1}\Lambda})"] 
&
D(X_{\pi^{-1}\Lambda,R}) \arrow[r,"D(\tilde{\rho}_{X^\vee, \pi^{-1}\Lambda})"] 
&
D(X^\vee)
\\
&
\mathrm{Pr}^{-1}(\calF)\arrow[u,hook]
&&
\operatorname{Fil}(X^\vee) \arrow[u,hook]
\end{tikzcd}.    
\]
Similar to Lemma \ref{lm 64}, for a $\bar{k}$-algebra $R$ and an $R$-point $(X, \iota, \lambda, \rho, \mathcal{F}) \in \mathcal{Y}^{\rm spl}(\Lambda)(R)$, we define
\[
\begin{array}{l}
U(X) := D(\tilde{\rho}_{X^\vee, \pi^{-1}\Lambda})^{-1}(\mathrm{Fil}(X^\vee))  / \mathrm{Im}(D(\rho_{\Lambda^{\sharp}, \pi^{-1}\Lambda})),  \\
U'(X) := D(\rho_{X^\vee, \pi^{-1}\Lambda})(\operatorname{Pr}^{-1}(\mathcal{F}))/\mathrm{Im}(D(\rho_{\Lambda^{\sharp}, \pi^{-1}\Lambda})).
\end{array}
\]
These are well-defined since the preimage $D(\tilde{\rho}_{X^\vee, \pi^{-1}\Lambda})^{-1}(\mathrm{Fil}(X^\vee))$ (resp. the image $D(\rho_{X^\vee, \pi^{-1}\Lambda})(\operatorname{Pr}^{-1}(\mathcal{F}))$ ) is a locally free direct summand of $D(X_{\pi^{-1}\Lambda,R})$ that contains
$\mathrm{Im}(D(\rho_{\Lambda^{\sharp}, \pi^{-1}\Lambda}))$. Since the proof is similar to Lemma \ref{lm 64}, we leave the details to the reader. 
Here, the quotient $U(X)$ (resp. $U'(X)$) is a locally free isotropic direct summand of rank $h-t$ (resp. $h-t-1$) and there are isomorphisms
\[
U(X)\simeq \Phi^{-1}(M^\sharp)/\breve{\Lambda}^\sharp,\quad
U'(X)\simeq M'/\breve{\Lambda}^\sharp.
\]
for any perfect field  $ \kappa$ over $\bar{k}$. Thus, by the chain of lattices $M'\subset M^\sharp\subset \Pi^{-1}M\subset \Pi^{-1}\Lambda$, we obtain
$(U(X), U'(X))\in (\mathrm{OGr}(h-t, V_{\Lambda^{\sharp}}) \times \mathrm{OGr}(h-t-1, V_{\Lambda^{\sharp}}))(\kappa)$ with $U'(X)\subset U(X)$.

\begin{Proposition}\label{BijYstrata}
Let $ \kappa$ be a perfect field over $\bar{k}$. There exists a bijective morphism $ f_{\calY}: \mathcal{Y}^{\rm spl}(\Lambda^\sharp)(\kappa) \longrightarrow R_{\Lambda^\sharp}'(\kappa)$ given by $(X, \iota, \lambda, \rho, \calF) \mapsto (U(X), U'(X))$.  
\end{Proposition}

\begin{proof}
We give a sketch of the proof since it similar to the proof of Lemma \ref{BijZstrata}. For a point $y \in \mathcal{Y}^{\rm spl}(\Lambda^{\sharp})$ we have $f_{\mathcal{Y}}(y) = (U, U')$ where
\[
(U, U') = (\Pi^{-1} V M^{\sharp} / \breve{\Lambda}^{\sharp}, M' /\breve{\Lambda}^{\sharp}) = (\tau^{-1}(M^{\sharp})/\breve{\Lambda}^{\sharp}, M' / \breve{\Lambda}^{\sharp}) = (\Phi^{-1}(M^{\sharp}/\breve{\Lambda}^{\sharp}), M'/\breve{\Lambda}^{\sharp}).
\]
To show that $f_{\mathcal{Y}}(y) \in R'_{\Lambda^{\sharp}}(\kappa)$ it suffices to prove that $U' \subset \Phi(U)\cap U $. 
Observe that conditions $M' \subset M^{\sharp} $ and $M'\subset \tau^{-1}(M^{\sharp}) $ are equivalent to $U' \subset \Phi(U) $ and $U' \subset U $ respectively. This shows that $f_{\mathcal{Y}}(y) \in R'_{\Lambda^{\sharp}}(\kappa)$.

Conversely, assume $(U, U') \in R'_{\Lambda^\sharp}(\kappa)$ and let $M^{\sharp} = \mathrm{Pr}^{-1}(\Phi(U))$ and $M' = \mathrm{Pr}^{-1}(U')$, where $\mathrm{Pr} :
\pi^{-1}\breve{\Lambda}\to \pi^{-1}\breve{\Lambda}/\breve{\Lambda}^{\sharp}$ is the natural projection map. Then by definition we have $\breve{\Lambda}^{\sharp}\subset M' \subset M^{\sharp} $. To show that $(M, M') \in \mathcal{Z}^{\mathrm{spl}}(\Lambda)(\kappa)$, it suffices, by the above, to show that $V M^{\sharp} \subset M'$. Observe that $V M^{\sharp} \subset V \Pi^{-1}\breve{\Lambda} = \ \breve{\Lambda} \subset \breve{\Lambda}^{\sharp} \subset M'$ and so $V M^{\sharp} \subset M'$. From the above we deduce that $f_{\mathcal{Y}}$ defines a bijection between $\mathcal{Y}^{\rm spl}(\Lambda^{\sharp})(\kappa)$ and $R'_{\Lambda^{\sharp}}(\kappa)$.
\end{proof}
\begin{Theorem}\label{Isom.Y}
The map $f_{\mathcal{Y}}:  \mathcal{Y}^{\rm spl}(\Lambda^{\sharp})\rightarrow R'_{\Lambda^{\sharp}} $ is an isomorphism.  
\end{Theorem}
\begin{proof}
Using the same proof as in Theorem \ref{Isom.X} we obtain that $f_{\mathcal{Y}}:  \mathcal{Y}^{\rm spl}(\Lambda^{\sharp})\rightarrow R'_{\Lambda^{\sharp}} $ is a closed immersion. Also, by Proposition \ref{Rsmooth}, $ R'_{\Lambda^{\sharp}}$ is irreducible with the same
dimension as $\mathcal{Y}^{\rm spl}(\Lambda^{\sharp})$. Therefore, $f_{\mathcal{Y}}$ is an isomorphism.     
\end{proof}

\begin{Corollary}
The BT-stratum $ \mathcal{Y}^{\rm spl}(\Lambda^{\sharp})$ is smooth and irreducible.    
\end{Corollary}
\begin{proof}
The result follows from Theorem \ref{Isom.Y} and Proposition \ref{Rsmooth}.    
\end{proof}

\begin{Remark}
{\rm
From \cite[Theorem 7.9]{HLS2}, there exists an isomorphism $\Phi_{\calY}: \calY^\loc(\Lambda^\sharp)\rightarrow R_{\Lambda^\sharp}$ given by $(X, \iota, \lambda, \rho) \mapsto U(X)$. We also have the following commutative diagram:
\[
\begin{tikzcd}
\calY^\spl(\Lambda^\sharp) \arrow[r,"f_\calY"]\arrow[d,"\mathrm{Pr}_1"]
&
R_{\Lambda^\sharp}'\arrow[d,"\mathrm{Pr}_2"]
\\
\calY^\loc(\Lambda^\sharp)\arrow[r,"\Phi_\calY"]
&
R_{\Lambda^\sharp}
\end{tikzcd},
\]
where $\mathrm{Pr}_1$ is given by $(X, \iota, \lambda, \rho, \calF)\mapsto (X, \iota, \lambda, \rho)$ and $\mathrm{Pr}_2$ is given by $(U, U')\mapsto U$.
}
\end{Remark}

\subsubsection{Intersection of $\calZ^\spl$-strata and $\calY^\spl$-strata}\label{IntZY}
We now discuss the intersection of $\calZ^\spl$-strata and $\calY^\spl$-strata. Let $\Lambda_1\subset \Lambda_2$ be vertex lattices with types $2t_1$ and $2t_2$, respectively, satisfying $2t_2<2h<2t_1$. We define the subvariety $R'_{[\Lambda_1, \Lambda_2]}\subset R'_{\Lambda_2^\sharp}$ whose $\bar{k}$-points are given by
\[
R'_{[\Lambda_1, \Lambda_2]}(\bar{k})=\left\{ (U, U')\in  R'_{\Lambda_2^\sharp} \mid U' \subset U\subset W \right\},
\]
where $W:=\Lambda_1^\sharp/ \Lambda_2^\sharp$ has dimension $t_1-t_2$. Equivalently,
{\small
\[
R'_{[\Lambda_1, \Lambda_2]}(\bar{k})= \left\{ (U, U') \in \left(\mathrm{OGr}(h-t_2, W_{\bar{k}}) \times \mathrm{OGr}(h-t_2-1, W_{\bar{k}})\right)(\bar{k}) \;\middle|\; U' \subset U \cap \Phi(U) \right\}.
\]
}
\begin{Proposition}
The projective variety $R'_{[\Lambda_1, \Lambda_2]}$ is irreducible and smooth of dimension $t_1-t_2-1$. 
\end{Proposition}

\begin{proof}
The proof is similar to that of Proposition \ref{Rsmooth}. Consider the reduced closed subvariety $S_{[\Lambda_1, \Lambda_2]}$ of  $ \mathrm{OGr}(h-t_2, V_{\Lambda_2^{\sharp},\bar{k}})$ whose $\bar{k}$-points are
\[
S_{[\Lambda_1, \Lambda_2]}(\bar{k}) = \left\{ U \in \mathrm{OGr}(h-t_2, V_{\Lambda_2^{\sharp},\bar{k}})(\bar{k}) \mid  U\subset W, \operatorname{dim} (U \cap \Phi(U)) \geq h-t_2-1 \right\}.
\]
Equivalently,  
\[
S_{[\Lambda_1, \Lambda_2]}(\bar{k}) = \left\{ U \in \mathrm{OGr}(h-t_2, W_{\bar{k}})(\bar{k}) \mid   \operatorname{dim} (U \cap \Phi(U)) \geq h-t_2-1 \right\}.
\]
There exists a forgetful morphism $\varphi_{[\Lambda_1, \Lambda_2]}:R'_{[\Lambda_1, \Lambda_2]}\rightarrow S_{[\Lambda_1, \Lambda_2]}$ given by $(U,U')\mapsto U$. By \cite[Proposition 6.11]{HLS2}, $S_{[\Lambda_1, \Lambda_2]}$ is irreducible and normal of dimension $t_1-t_2-1$. Using the same method as in the proof of Proposition \ref{Rsmooth}, we deduce that $R'_{[\Lambda_1, \Lambda_2]}$ is irreducible, smooth and of dimension $t_1-t_2-1$.
\end{proof}

\begin{Theorem}\label{Thm 613}
The restriction of the morphism $f_{\mathcal{Y}}:  \mathcal{Y}^{\rm spl}(\Lambda_2^{\sharp})\rightarrow R'_{\Lambda_2^{\sharp}} $ to the intersection $\calZ^\spl(\Lambda_1)\cap \mathcal{Y}^{\rm spl}(\Lambda^{\sharp}_2)$ defines an isomorphism
\[
f_{\calZ \cap\calY}: \calZ^\spl(\Lambda_1)\cap \mathcal{Y}^{\rm spl}(\Lambda^{\sharp}_2)\rightarrow R'_{[\Lambda_1, \Lambda_2]}.
\]
\end{Theorem}

\begin{proof}
For a point $(X, \iota, \lambda, \rho, \mathcal{F}) \in \calZ^\spl(\Lambda_1)\cap \mathcal{Y}^{\rm spl}(\Lambda^{\sharp}_2)(\bar{k})$, we have a chain of lattices
\[
\breve{\Lambda}_2^\sharp \subset M(X)'\subset M(X)^\sharp\subset \breve{\Lambda}_1^\sharp\subset \Pi^{-1}\breve{\Lambda}_1\subset \Pi^{-1}\breve{\Lambda}_2,
\]
which implies
\[
U'(X)\subset \Phi(U(X))\subset W\subset (\Pi^{-1}\breve{\Lambda}_1)/ \breve{\Lambda}_2^\sharp\subset V_{\Lambda_2^\sharp}.
\]
Note that $W$ is $\Phi$-invariant and we can easily see that $U' \subset U\subset W$. Thus we have the restriction morphism $f_{\calZ \cap\calY}: \calZ^\spl(\Lambda_1)\cap \mathcal{Y}^{\rm spl}(\Lambda^{\sharp}_2)\rightarrow R'_{[\Lambda_1, \Lambda_2]}$. 

Conversely, for a point $(U, U')\in  R'_{[\Lambda_1, \Lambda_2]}$, we define $M^{\sharp} = \mathrm{Pr}^{-1}(\Phi(U))$ and $M' = \mathrm{Pr}^{-1}(U')$. To check $(M, M')\in \calZ^\spl(\Lambda_1)(\bar{k})$, it suffices to check $\breve{\Lambda}_1\subset M$ and $\breve{\Lambda}_1\subset M'$ by Proposition \ref{BijZstrata}. Observe that $M^\sharp\subset \breve{\Lambda}_1^\sharp$ by $U\subset W$, and $\breve{\Lambda}_1\subset \breve{\Lambda}_2\subset \breve{\Lambda}_2^\sharp\subset M'$. This finishes the proof of the theorem.
\end{proof}

\Addresses
\end{document}